\documentclass[]{article}
\usepackage[letterpaper,top=2cm,bottom=2cm,left=3cm,right=3cm,marginparwidth=1.75cm]{geometry}
\usepackage{hyperref}
\usepackage{enumitem} 
\usepackage{amsmath}
\usepackage{amsfonts}
\usepackage{amsthm}
\usepackage{amssymb}
\usepackage{tikz-cd}
\usepackage{mathrsfs}
\usepackage{mathtools}
\usepackage{textcomp}
\theoremstyle{plain}
\newtheorem{thm}{Theorem}[section]
\newtheorem{prop}[thm]{Proposition}
\newtheorem{cor}[thm]{Corollary}
\newtheorem{lemma}[thm]{Lemma}
\theoremstyle{definition}
\newtheorem{defn}[thm]{Definition}

\theoremstyle{remark}
\newtheorem{rmk}[thm]{Remark}
\newtheorem{notation}[thm]{Notation}
\title{Irreducibility of Newton strata in Picard modular surfaces and split local Galois representations}
\author{Haocheng Fan}
\newcommand{\QQ}{\mathbb{Q}}
\newcommand{\ZZ}{\mathbb{Z}}
\newcommand{\RR}{\mathbb{R}}

\newcommand{\CC}{\mathbb{C}}
\newcommand{\OO}{\mathcal{O}}

\newcommand{\GG}{\mathbf{G}}
\newcommand{\GSp}{\mathbf{GSp}}
\newcommand{\PP}{\mathbf{P}}
\newcommand{\MM}{\mathbf{M}}
\newcommand{\Gm}{\mathbb{G}_\mathbf{m}}
\newcommand{\T}{\mathbf{T}}
\newcommand{\XX}{\mathbf{X}}
\newcommand{\PPP}{\mathcal{P}}
\newcommand{\AAA}{\mathcal{A}}
\newcommand{\FF}{\mathbb{F}}
\newcommand{\E}{\mathcal{E}}
\newcommand{\etale}{\'etale}

\newcommand{\m}{\mathbf{m}}
\newcommand{\et}{\acute{e}t}
\newcommand{\II}{II}

\DeclareMathOperator{\coker}{coker}

\date{}
\begin{document}

\maketitle
\numberwithin{equation}{section}

\begin{abstract}
    We show that for a Picard modular form, the existence of companion forms is equivalent to the splitting properties of the associated local Galois representation. This result is obtained by using the computation of the monodromy group and the irreducibility for the closure of the non-ordinary Newton stratum in the special fiber of the Picard modular surface at a split prime.
\end{abstract}

\section{Introduction}
    In this paper, we find some analogs of the theorem due to Breuil and Emerton \cite[Thm. 1.1.3]{breuil2010representations}, which states that a $p$-ordinary cuspform $f$ has a companion form $g$ if and only if the attached $p$-adic representation is split. Our approach follows the original idea of Breuil--Emerton, which views the $p$-adic representation as the one cut out by the Hecke eigensystem in $p$-adic cohomology groups of geometric objects. However, in our case, the geometric objects are more complicated, so we need to analyze the monodromy group of the closed Newton stratum to compute some cohomology groups. This kind of result on the monodromy groups is related to Oort's Hecke orbit conjecture in positive characteristics and proved by Ching-Li Chai \cite{chai2005monodromy} in the Siegel case and Pol van Hoften \cite{van2024ordinary, van2024mod} in the Hodge type cases. We include a proof in our specific case in section 4 to clarify the core ideas and ensure a self-contained account.

    We introduce Breuil--Emerton's Theorem first. Let $N\ge4$ be an integer and $p$ be an odd prime such that $(p,N)=1$. Assume that $f$ is a cuspidal eigenform of weight $k\ge 2$, level $\Gamma_1(N)$, and character $\chi$. It is well known that there is a $2$-dim Galois representation $\sigma(f)$ over $\overline{\mathbb{Q}}_p$ associated with $f$, and let $\sigma_p(f)=\sigma(f)|_{\text{Gal}(\overline{\mathbb{Q}}_p/\mathbb{Q}_p)}$ be the restriction. Let $D_p:=D_{p,cris}(\sigma_p(f))$ be the associated $2$-dim filtered $\varphi$-module. The eigenvalues of $\varphi$ in $D_p$ are the two roots $\alpha_p,\beta_p$ of the polynomial
	$$  
        X^2-a_p(f)+p^{k-1}\chi(p)=0,
    $$
    where $a_p(f)$ is the eigenvalue of Hecke operator $T_p$ on $f$. It is not difficult to find that $\sigma_p(f)$ is reducible if and only if $v_p(\alpha_p)=k-1$ and $v_p(\beta_p)=0$. The following question is under what condition $\sigma_p(f)$ is split.
    
    Breuil and Emerton provided a method for attaching split local representations to companion forms. A key tool is the operator $\theta^{k-1}$ in the space of overconvergent forms. Let $X_1(N)$ be the compact modular curve of level $\Gamma_1(N)$, and $\pi:E\to X_1(N)$ be the universal elliptic curve. Let $C$ be the cusps on $X_1(N)$ and $\tilde{C}=\pi^{-1}(C)$. We denote 
    $$
        \mathcal{H}_{k-2}:=\text{Sym}^{k-2}(\mathcal{H}_{log-dR}^1(E/X_1(N))), \ \  \omega:=\pi_*\Omega^1_{E/X_1(N)}(log\tilde{C}). 
    $$ The Gauss--Manin connection gives a map
    $$
        \nabla:\mathcal{H}_{k-2}\to\mathcal{H}_{k-2}\otimes\Omega^1(logC)
    $$ 
    and induces a map 
    $$
        \theta^{k-1}:\omega^{2-k}\to{\omega^k}.
    $$
    Thus, there is an operator on overconvergent modular forms:
    $$
        \theta^{k-1}:S^{\dagger}_{2-k}(\Gamma_1(N))\to S^{\dagger}_k(\Gamma_1(N)).
    $$

\begin{thm}
    The local representation $\sigma_p$ is split if and only if there exists an overconvergent modular form $g\in S^{\dagger}_{2-k}(\Gamma_1(N))$ such that $\tilde{f}:=f(z)-\beta_pf(pz)=\theta^{k-1}(g)$.
\end{thm}

\begin{rmk}
    \begin{enumerate}[label=(\arabic*)]
        \item The level $pN$ modular form $\tilde{f}$ is called the $\alpha_p$-stablization of $f$, which is an eigenform with $U_p$ eigenvalue $\alpha_p$ and with same $T_l$ eigenvalues as $f$.

        \item As $\theta^{k-1}$ commutes with $T_l$, the overconvergent form $g$ has same $T_l,S_l$ eigenvalues as $f$, hence the Galois representation associated with $g$ is again $\sigma(f)$. This $g$ is called a companion form of $f$.
    \end{enumerate}
     
\end{rmk}

	Now we introduce our analog on Picard modular forms and the associated $3$-dimensional representations. Let $E/\QQ$ be an imaginary quadratic field and $p=v\bar{v}$ an odd prime that splits in $E$. Let $\PPP_K$ denote the integral model of the Picard modular surface (i.e. a Shimura varatiy of $GU(2,1)$) over $R_1=\OO_{E,v}$ with a neat level $K$ that is hyperspecial at $p$, and $P_K/E$ (resp. $\overline{P}_K/\FF_p$) denote its generic fiber (resp. special fiber). We first apply the main theorem of \cite{BGG} to compute the dual BGG complex over $P_K$:
	\begin{prop}(Proposition \ref{BGG})
		Assume that $R$ is an $E$-algebra and $(\underline{k})=(k_1,k_2,k_3,k_4)$ is a dominant weight. Then the natural embedding
		\begin{align*}
		&BGG^{\bullet}(\mathcal{F}^{(\underline{k})}_R):=[\omega^{(k_2+1,k_3-2,k_1+1,k_4)}_R
		\to\omega^{(k_1,k_3-2,k_2+2,k_4)}_R \to \omega^{(k_1,k_2,k_3,k_4)}_R ]      \notag\\
		\hookrightarrow &DR^{\bullet}(\mathcal{F}^{(\underline{k})}_R):=[\mathcal{F}^{(\underline{k})}_R \to \mathcal{F}^{(\underline{k})}_R \otimes {\Omega}^{1}_{\PPP_R} \to \mathcal{F}^{(\underline{k})}_R \otimes {\Omega}^{2}_{\PPP_R}]
		\end{align*} is a quasi-isomorphism, and the similar result holds for canonical and subcanonical extensions.
	\end{prop} 
    We refer to section $3$ for the detailed definition of the automorphic bundles that appear in the previous proposition. They can be viewed as analogs of the line bundle $\omega^k$, and their global sections are defined as Picard modular forms. 
    For any weight $(\underline{k})$ that $k_1\leq k_2$, we define the space of Picard modular forms of level $K$ with coefficients in $R$ by
	$$
	M_{(\underline{k})}(K,R):=H^0(\PPP_{K,R}^{tor}, \omega_R^{(\underline{k}),can}),
	$$
    where $\PPP_{K,R}^{tor}$ is the base change of the toroidal compactification of $\PPP_{K}$ to the $E$-algebra $R$.
	The subspace of cusp forms is defined as
	$$
	S_{(\underline{k})}(K,R):=H^0(\PPP_{K,R}^{tor}, \omega_R^{(\underline{k}),sub}).
	$$
	Let $\Theta:\omega^{(k_1,k_3-2,k_2+2,k_4)}_R \to \omega^{(k_1,k_2,k_3,k_4)}_R$ denote the differential operator in $BGG^{\bullet}(\mathcal{F}^{(\underline{k})}_R)$. Then it induces a linear map 
	$$
		\Theta: S^{\dagger}_{(\underline{k})}(K,L) \to S^{\dagger}_{(k_1,k_3-2,k_2+2,k_4)}(K,L)
	$$
	between overconvergent cuspidal forms with coefficients in $L$ for any finite extension $L/\QQ_p$ and weight $(\underline{k})$ that satisfies $k_2+3\le k_3$. The overconvergent forms are the global sections of the automorphic vector bundles on a tube of the ordinary locus, and we refer to Section 5 for details.

	Now we state the results on the monodromy groups we need. Fix a prime $l \ne p$ and a neat level $K=K_lK^l$ that is hyperspecial at $p$ and $K_l\subset \GG(\mathbb{Z}_l)$, where $\GG$ is the reductive group in the Shimura datum. We consider the Newton strata on $\overline{\PPP}_{K,\overline{\FF}_p}$. It coincides with the Ekedahl-Oort strata in our case, and there are three possible Newton polygons. Let $Newt$ be the set of Newton polygons, and for $\xi\in Newt$, let $W_{\xi}$ denote the corresponding Newton stratum. Let $\overline{\PPP}_{K}^{0}$ be a connected component of $\overline{\PPP}_{K,\overline{\FF}_p}$. Following the strategy of \cite{chai2005monodromy} and \cite{2006Monodromy}, we prove 
	
	\begin{thm}(Theorem \ref{connected})
		For any non-basic $\xi \in Newt$, we have $Z_{\xi}=W_{\xi} \cap \overline{\PPP}_{K}^{0}$ is connected, hence irreducible as it is smooth. For the monodromy representation $\rho_l: \pi_1^{\et}(Z_{\xi},\bar{z}) \to \pi_1^{\et}(\overline{Z_{\xi}},\bar{z}) \to \GG(\ZZ_l)$ associated with the $l$-adic Tate module of the universal abelian scheme, the images of both fundamental groups in $\GG(\ZZ_l)$ are $K_{l}\cap \GG'(\mathbb{Q}_l)$, where $\GG'(\mathbb{Q}_l)$ is the derived group of $\textbf{G}(\mathbb{Q}_l)$.
	\end{thm}
    \begin{rmk}
        This is not a new result, as Pol van Hoften has proved the connectedness of the Kottwitz–Rapoport strata and studied the monodromy groups of Hecke-equivariant subvarieties in Shimura varieties of Hodge type with parahoric level, see \cite{van2024mod} and \cite{van2024ordinary}. We include a proof in our specific case in section 4 to ensure self-containment.
    \end{rmk}

	Applying the above theorem, we can compute some cohomology groups and consider a splitting property of the Galois representation associated with a cuspidal form. Assume that $L/\QQ_p$ is a finite extension and $f\in S_{(\underline{k})}(K, L)$ is a cuspidal eigen-newform (i.e. $f$ is an eigenform and $\dim (\pi^{\infty})^{K}=1$, where $\pi(f)=\pi^{\infty}\otimes\pi_{\infty}$ is the associated cuspidal automorphic representation). Let $\mathfrak{m}$ denote the maximal ideal $\mathscr{H}(K^p,L)$ associated with the Hecke eigensystem of $f$. Assume that the associated Galois representation $\sigma(f)=(H^2_{c,\et}(P_{K,\overline{E}}, \mathscr{L}^{(\underline{k})})\otimes_{\QQ_p}L)[\mathfrak{m}]$ is irreducible of dimension 3. We write $\sigma_p=\sigma(f)|_{{\rm Gal}(\overline{\QQ}_p/\QQ_p)}$ and $D_p=D_{cris}(\sigma_p)$. We further assume that the $\varphi$-operator of $D_p$ has three different eigenvalues $\alpha_1,\alpha_2,\alpha_3 \in L$  (It makes sense as $\varphi$ is linear in this case) such that $v_p(\alpha_1)\le v_p(\alpha_2)\le v_p(\alpha_3)$. A weight $(\underline{k})$ is called \textit{cohomological} if $k_2+3\leq k_3$ and \textit {good} if $k_2+3\leq k_3\leq 2$ and $k_4=0$. 
	
	\begin{prop}(Proposition \ref{injectivity})
		For any good weight $(\underline{k})$ and finite extension $L$ over $\QQ_p$, we have the commutative diagram
		$$
		\begin{tikzcd}	 	
		S_{(\underline{k})}(K,L)[\m]  \arrow[d,hook] \arrow[r,hook]
		&  \dfrac{S^{\dagger}_{(\underline{k})}(K,L)[\m]}{\Theta(S^{\dagger}_{(k_1,k_3-2,k_2+2,0)}(K,L)[\m])} 
		\\
		\mathbb{H}^2(P^{tor}_{K,L}, DR^{\bullet}(\mathcal{F}_L^{(\underline{k})})^{sub})[\m] \arrow[ru,hook,"i"]
		&  
		\end{tikzcd}.
		$$ The inclusion $i$ is $\varphi$-equivariant, where the $\varphi$ actions come from the Frobenius operators on rigid cohomology groups.
	\end{prop}

	Twisting central characters, it is obvious that the previous proposition holds for arbitrary cohomological weight $(\underline{k})$. The vertical arrow on the left is exactly the inclusion $Fil^{k_3}D_p\subset D_p$. Therefore, for a good weight $(\underline{k})$, we have
	\begin{thm}(Theorem \ref{phi-stable})
		$Fil^{k_3}D_p$ is $\varphi$-stable if and only if $\varphi f-\alpha_3 f= \Theta(g)$ for some cuspidal overconvergent form $g\in S^{\dagger}_{(k_1,k_3-2,k_2+2,0)}(K,L)[\m]$, moreover,
		$$
		k_1+1 \leq v_p(\alpha_1) \leq v_p(\alpha_2)\leq k_2+2<v_p(\alpha_3)=k_3,
		$$
		hence $D_3$ and $D_{12}$ are admissible. Here we view $f$ as an overconvergent form naturally, so $\varphi f$ makes sense.
	\end{thm}
	Then we have the following corollary for cohomological weights $(\underline{k})$ immediately. This is our analog of Breuil--Emerton's theorem.
	\begin{thm}(Corollary \ref{split})
		Assume that $f\in S_{(\underline{k})}(K,L)$ is of cohomological weight $(\underline{k})$. Then the crystalline Galois representation $\sigma_p$ whose Hodge-Tate weight is $(k_1-k_4+1,k_2-k_4+2,k_3-k_4 )$ splits as $\sigma_p=\sigma_{p,12}\oplus\sigma_{p,3}$, whose Hodge-Tate weights are $(k_1-k_4+1,k_2-k_4+2)$ and $(k_3-k_4)$ respectively, if and only if $(\varphi-\alpha_3)f=\Theta(g)$ for some $g\in S^{\dagger}_{(k_1,k_3-2,k_2+2,k_4)}(K,L)[\m]$. In particular, $g$ is a companion form of $f$.
	\end{thm}

\vspace{0.3cm}

\noindent\textbf{Acknowledgements}. The author thanks Yiwen Ding and Liang Xiao for asking me to consider this problem. The author also thanks Ruiqi Bai, Tian Qiu, and Deding Yang for helpful discussions.

\section{Picard modular surface}

	In this section, we review the definition of the Picard modular surface and its moduli problem and do careful linear algebra preparations to satisfy \cite[Condition 2.5]{BGG} because we want to apply the results about the dual BGG complex due to K.-W. Lan and P. Polo in \cite{BGG}.
	
	\begin{notation}
	We assume that 
	\begin{itemize}
		\item $E$ is a imaginary quadratic field over $\QQ$, and $\bar{\cdot}$ is the complex conjugate. \item $\iota:E \to \CC$ is an embedding and $\bar{\iota}$ is the complex conjugate of $\iota$. We view $E$ as a subfield of $\CC$ via $\iota$.
		\item $d_E$ is a square free integer such that $E=\QQ[\sqrt{d_E}]$, and $D_E$ is the discriminant of $E$. 
		\item $\delta$ is a square root of $D_E$ such that  $\iota(\delta)$ has positive imaginary part. ${\rm Im}_{\delta}(a)=(a-\bar{a})/\delta$.
		\item $p$ is an odd prime number. Assume that $p=v\bar{v}$ is split in $E$. Further $p \nmid d_E$.
		\item $R_1=\OO_{E,v}$, the localization of $\OO_E$ at $v$, it is our base ring.
	\end{itemize}
	\end{notation}

	Let $(V,(\cdot,\cdot))$ be a hermitian space, $V=E^3$ with standard basis $e_1,e_2,e_3$, the hermitian pairing $(\cdot,\cdot): V \times V \to E$ is given by
	$$
	(x,y)=\bar{x}^{t} \left(
	  \begin{array}{ccc}
		 1 &   &   \\ 
		   & 1 &   \\ 
		   &   & -1 
	\end{array} \right) y.
	$$
	We consider the alternating pairing $\left\langle \cdot , \cdot \right\rangle: V \times V \to \QQ$, given by $\left\langle x,y \right\rangle ={\rm Im}_\delta(x,y)$. Then we have the formulae
	$$
	\left\langle ax,y\right\rangle =\left\langle x,\bar{a}y \right\rangle, \quad
		2(u,v)=\left\langle u,\delta v\right\rangle + \delta\left\langle u,v \right\rangle .
	$$
	We fix a lattice $L=Span_{\OO_E}\{e_1,e_2,e_3\}$ in $V$. It is a self-dual lattice in the sense that
	$$
		L=\left\lbrace x\in V \mid \left\langle x,y \right\rangle \in \ZZ , \forall y \in L \right\rbrace.
	$$ 
	Equivalently, $L$ is the $\OO_E$-dual of itself with respect to the hermitian pairing $(,)$.

	We have viewed $E$ as a subfield of $\CC$, so we can identify $\OO_E \otimes_{\ZZ} \RR$ and $\CC$, and then identify $L_{\RR}:=L \otimes_{\ZZ} \RR=V_{\RR}$ and $\CC^3$. Let the $\RR$-homomorphism $h_0$ be
	\begin{align*}
		h_0:\CC & \to {\rm M}_3(\CC) = {\rm End}_{\OO_E \otimes_{\ZZ} \RR}(L_{\RR}) \\
		z & \mapsto \begin{pmatrix}
		z &  &  \\ 
		& z &  \\ 
		&  & \bar{z}
		\end{pmatrix} .
	\end{align*}
	Then it is not difficult to check that $(\OO_E,\bar{\cdot},L,\left\langle \cdot,\cdot \right\rangle ,h_0)$ is an \emph{integral} PEL-datum, in the sense of \cite{BGG}. It is an integral version of the data $(E, \bar{\cdot}, V, \left\langle \cdot,\cdot \right\rangle ,h_0))$.
	
	The homomorphism $h_0$ defines a Hodge structure on $L$. More precisely, it defines an action of $(\CC\otimes_{\RR}\CC)^{\times}=\CC^{\times} \times \CC^{\times}$ on $L_{\CC}:=L_{\RR}\otimes_{\RR}\CC$, where $z_1\otimes z_2$ acts by $h_0(z_1)\otimes z_2$. Then we have the \textit{Hodge decomposition} $L_\CC=V_0 \oplus \overline{V_0}$, such that $h_0(z)\otimes 1$ acts by $1\otimes z$ on $V_0$ and by $1\otimes \bar{z}$ on $\overline{V_0}$. The complex conjugation here sends $v\otimes z$ to $v\otimes \bar{z}$.
	
	It is not difficult to check that $V_0$ is the maximal totally isotropic subspace under the pairing $\left\langle ,\right\rangle $. Then the pairing induces an isomorphism 
	\begin{align}\label{iso_sp}
		\overline{V_0}\cong V_0^{\vee}(1):={\rm Hom}_{\CC}(V_0,\CC)(1).
	\end{align}
	There is a canonical pairing $\left\langle , \right\rangle_{can}$ on $V_0 \oplus V_0^{\vee}(1)$, which is defined by
	\begin{align*}
		\left\langle (x_1,f_1) , (x_2,f_2) \right\rangle_{can}=f_2(x_1)-f_1(x_2),
	\end{align*}
	and the isomorphism (\ref{iso_sp}) induces
	\begin{align} \label{iso_pair}
	(V_0\oplus\overline{V_0},\left\langle , \right\rangle ) \cong (V_0 \oplus V_0^{\vee}(1) , \left\langle , \right\rangle_{can}).
	\end{align}
	
	We consider the free $\OO_{E,(p)}$ module $L\otimes_{\ZZ}\OO_{E,(p)}$ in $L_\CC$, and we can check that  $L\otimes_{\ZZ}\OO_{E,(p)}=L_0\oplus \overline{L_0}$, where $L_0\subset V_0$ is the free $\OO_E\otimes_{\ZZ}\OO_{E,(p)}$ module 
	\begin{align*}
		L_0&=Span_{\OO_E\otimes_{\ZZ}\OO_{E,(p)}}\{v_1,v_2,v_3\} \\
		v_1&=e_1\otimes 1 + \delta e_1 \otimes \delta^{-1} \\
		v_2&=e_2\otimes 1 + \delta e_2 \otimes \delta^{-1} \\
		v_3&=e_3\otimes 1 - \delta e_3 \otimes \delta^{-1}	
	\end{align*}
	and $L_0\otimes_{\OO_{E,(p)}}\CC=V_0$. Recall that $L$ is a self-dual lattice, so the isomorphism $\overline{V_0}\cong V_0^{\vee}(1)$ sends $\overline{L_0}$ to $L_0^{\vee}(1)$. Hence, (\ref{iso_pair}) induces 
	\begin{align} \label{iso_lattice}
		(L\otimes_{\ZZ}\OO_{E,(p)},\left\langle , \right\rangle ) \cong (L_0 \oplus L_0^{\vee}(1) , \left\langle , \right\rangle_{can}).
	\end{align}
	Now it is obvious that $R_1$ and $L_0$ satisfy \cite[Cond. 2.5]{BGG}, and both the reflex field $F_0$ and its extension $F_0'$ in their settings are equal to $E$ in our case.
	
	There is also a $\textit{type decomposition}$ of the $\OO_E\otimes_{\ZZ}\OO_{E,(p)}$ modules $L_0$ and $L_0^{\vee}(1)$. Let $\tau: \OO_E \to \OO_{E,(p)}$ be the natural embedding (and $v$ correspond to $\tau$ through a fixed isomophism $\iota_p:\overline{\QQ}_p\cong \CC$), and $\bar{\tau}$ be the complex conjugate of $\tau$. Then $L_0=L_{0,\tau} \oplus L_{0,\bar{\tau}}$, where $a\otimes1$ acts by $1\otimes\tau(a)$ on $L_{0,\tau}$ and by $1\otimes\bar{\tau}(a)$ on $L_{0,\bar{\tau}}$. More explicitly, $L_{0,\tau}$ is spanned by $v_1,v_2$ and $L_{0,\bar{\tau}}$ is spanned by $v_3$. Similarly $L_0^{\vee}(1)_{\tau}$ is spanned by $v_3^{\vee}$ and $L_0^{\vee}(1)_{\bar{\tau}}$ is spanned by $v_1^{\vee},v_2^{\vee}$.
	
	\begin{defn}
		Let $(L,\left\langle , \right\rangle )$ be as above, we define for each $\ZZ$-algebra $R$
		\begin{align*}
			\GG(R):=\left\lbrace 
			\begin{array}{l}
			(g,r)\in {\rm Aut}_{\OO_E\otimes_{\ZZ}R}(L\otimes_{\ZZ}R) \times \Gm(R)  \\ 
			\mid \left\langle gx,gy \right\rangle = r\left\langle x,y \right\rangle , \forall x,y  \in L\otimes_{\ZZ} R
			\end{array}   
			\right\rbrace 
		\end{align*}.
	\end{defn}
	The group $\GG(R)$ is functorial in $R$, so $\GG$ is a group functor over $\ZZ$. We write $G_\infty=\GG(\RR) \cong GU(2,1)$, $G_p=\GG(\QQ_p) \cong GL_3(\QQ_p)\times \QQ_p^{\times}$.
	
	Let $\mu:\GG\to\Gm$ be the similitude character sending $(g,r)$ to $r$, and ${\rm det}: \GG \to \T={\rm Res}_{\OO_E/\ZZ}\Gm$ be the determination character sending $(g,r)$ to ${\rm det}(g)$. Let $\rho$ be the non-trivial automorphism on $\T$ and $\nu=\mu^{-1}\cdot{\rm det}$. Then $\mu=\nu\cdot(\rho\circ\nu), {\rm det}=\nu^2\cdot(\rho\circ\nu)$. The unitary and special unitary subgroups of $\GG$ are defined as
	\begin{align*}
		\mathbf{U}={\rm ker}\mu \quad \mathbf{SU}={\rm ker}\nu=\GG'.
	\end{align*}
	
	\begin{defn}
	Let $(L_0\oplus L_0^{\vee}(1), \left\langle ,\right\rangle _{can})$ be the same as above. We define the following groups for any $\OO_{E,(p)}$-algebra $R$ as
		\begin{align*}
			\GG_0(R) &:=\left\lbrace 
			\begin{array}{l}
			(g,r)\in {\rm Aut}_{\OO_E\otimes_{\ZZ}R}((L_0\oplus L_0^{\vee}(1))\otimes_{\OO_{E,(p)}}R) \times \Gm(R)  \\ 
			\mid \left\langle gx,gy \right\rangle_{can} = r\left\langle x,y \right\rangle_{can} , \forall x,y \in (L_0\oplus L_0^{\vee}(1)\otimes_{\OO_{E,(p)}} R 
			\end{array}
			\right\rbrace , \\
			\PP_0(R) &:=\left\lbrace
			(g,r)\in \GG_0(R) \mid g(L_0^{\vee}(1)\otimes_{\OO_{E,(p)}}R)=L_0^{\vee}(1)\otimes_{\OO_{E,(p)}}R
			\right\rbrace  , \\
			\MM_0(R) &:={\rm Aut}_{\OO_E\otimes_{\ZZ} R}(L_0^{\vee}(1)\otimes_{\OO_{E,(p)}}R)\times \Gm(R).
		\end{align*}
		The groups $\GG_0(R),\PP_0(R),\MM_0(R)$ are functorial in $R$, so $\GG_0,\PP_0,\MM_0$ are group functors over $\OO_{E,(p)}$. 
	\end{defn}

	The isomorphism (\ref{iso_lattice}) induces $\GG\otimes_{\ZZ}\OO_{E,(p)} \cong \GG_0$, which is a reductive split group scheme, the group scheme $\PP_0$ is a parabolic subgroup scheme and $\MM_0$ is a split Levi quotient of $\PP_0$. The splitting map $\MM_0(R) \to \PP_0(R)$ is defined by sending $(m,r)$ to $(g,r)$, which is the unique element in $\GG_0(R)$ that $g\mid_{L_0^{\vee}(1)_R}=m$. 
	
	We fix the maximal split torus $\T_0$ of $\GG_0$ by choosing an embedding $\Gm^4(R) \to \GG_0(R)$ for any $\OO_{E,(p)}$-algebra $R$. The lattice $L_{\tau,R}=(L_0\oplus L_0^{\vee}(1))_{\tau,R}$ is a free module spanned by $v_1,v_2,v_3^{\vee}$. The splitting map $\Gm^4(R) \to \GG_0(R)$ is defined by sending $(t_1,t_2,t_3,1)$ to the diagram action on $\{v_1,v_2,v_3^{\vee}\}$ and sending $(1,1,1,r)$ to $r\cdot {\rm Id}_{L_{\tau,R}}$.

	Let $h:\mathbb{S}(\RR)\cong\CC^{\times}\to\GG(\RR)$ be the homomorphism induced by $h_0$, the group $K_{\infty}$ be the stabilizer of $h$, and $\XX=G_{\infty}/K_{\infty}$ be the space of the $\GG(\RR)$-conjugacy class of $h$. The space $\XX$ has a complex manifold structure. The pair $(\GG_{\QQ}, \XX)$ forms a Shimura datum with a reflex field $E$. And the embedding of Shimura data $\varphi:(\GG_{\QQ},\XX)\hookrightarrow (GSp_6,\XX_6)$ is naturally defined by forgetting the $E$-action, and $\varphi$ extends to an embedding of the $\ZZ$-group schemes $\varphi: \GG \hookrightarrow \GSp=GSp(L,\left\langle,\right\rangle )$.
	\begin{defn}	
		For any open compact subgroup $K\subset\GG(\mathbb{A}_f)$, we have the Shimura variety $Sh_{K}$ with complex points 
		\begin{align*}
		Sh_{K}(\CC)=\GG(\QQ)\backslash \XX \times \GG(\mathbb{A}_f)/K
		\end{align*}
		and admits a canonical model over $E$.
	\end{defn}
	
	\begin{rmk}
		The subgroup $K_{\infty}$ is isomorphic to $G(U(2)\times U(1))$ as a subgroup of $G_{\infty}\cong GU(2,1)$, and the space $\XX$ is biholomorphic to the open unit ball in $\CC^2$. We refer to section 1.2 in \cite{picard} for more details of the complex manifold structure of the above Shimura variety and the universal abelian variety. Our choice of $L$ and $(,)$ are different in \cite{picard} but it is not hard to translate the results.
	\end{rmk}	
	
	\begin{defn}\label{integral model}
		(\cite[Def. 1.4.1.4]{PEL}, $\square=\{p\}$) Let $K^p\subset\GG(\hat{\ZZ}^p)$ be a neat open compact subgroup, $K_p$ the hyperspecial subgroup $\GG(\ZZ_p)$, and $K=K^p K_p$. We consider the moduli problem $\PPP_{K}$ over $Spec(\OO_{E,(p)})$, such that for any scheme $S$ over $Spec(\OO_{E,(p)})$, $\PPP_{K}(S)$ parameterizing the isomorphism classes of tuples $(A,\lambda, i, \alpha)$ as follows:
		\begin{enumerate}
			\item $A \to S$ is an abelian scheme of relative dimension 3.
			\item $\lambda: A\to A^{\vee}$ is a principal polarization.
			\item $i:\OO_E\to {\rm End}_{S}(A)$ is a homomorphism such that the complex conjugation in $\OO_E$ is sent to the Rosati involusion induced by $\lambda$.
			\item (signature condition) The characteristic polynomial of $i(a)$ in $Lie_{A/S}$ is $(x-a)^2(x-\bar{a})$.
			\item $\alpha$ is an (integral) level-$K^p$ structure of $(A,\lambda,i)$ of type $(L\otimes_{\ZZ}\hat{\ZZ}^p, \left\langle \right\rangle )$, in the sense of \cite[Def. 1.3.7.6]{PEL}.
		\end{enumerate}
	\end{defn}
	
	The definition here is slightly different from \cite[Def. 1.4.1.4]{PEL} and they are equivalent. In our case $L$ is a self-dual lattice, so the existence of $\alpha$ implies that $\lambda$ is a principal polarization. The signature condition is equivalent to the {\rm determinantal condition}. Therefore, the two moduli problems are equivalent, and by \cite[Thm. 1.4.1.11]{PEL} and \cite[Cor. 7.2.3.10]{PEL} we see that $\PPP_{K}$ is represented by a smooth quasi-projective scheme over $Spec(\OO_{E,(p)})$.
	
	\begin{defn}
		(\cite[Def. 1.4.1.4]{PEL}, $\square=\emptyset$) Let $K\subset\GG(\hat{\ZZ})$ be a neat open compact subgroup. We consider the moduli problem $P_{K}$ over $Spec(E)$, such that for any scheme $S$ over $Spec(E)$, $P_{K}(S)$ parameterizing the isomorphism classes of tuples $(A,\lambda, i, \alpha)$ as follows:
		\begin{enumerate}
			\item $(A,\lambda,i)$ satisfies the same conditions as Definition \ref{integral model}.
			\item $\alpha$ is an (integral) level-$K$ structure of $(A,\lambda,i)$ of type $(L\otimes_{\ZZ}\hat{\ZZ}, \left\langle \right\rangle )$, in the sense of \cite[Def. 1.3.7.6]{PEL}.
		\end{enumerate}
	\end{defn}
	
	Applying \cite[Thm. 1.4.1.11]{PEL} and \cite[Cor. 7.2.3.10]{PEL} once again, we see that $P_{K}$ is represented by a smooth quasi-projective scheme over $Spec(E)$. For any neat $K\subset\GG(\hat{\ZZ})$, $P_{K}$ is exactly the canonical model of $Sh_{K}$ over $E$. Moreover, if $K_p$ is hyperspecial, $\PPP_{K}$ is an integral model of $P_{K}$.
	
	Assume that $K\subset\GG({\hat{\ZZ}})$ is neat and $K_p$ is hyperspecial from now on. Let $\widetilde{K}\subset \GSp(\hat{\ZZ})$ be a neat subgroup, hyperspecial at $p$, and $\varphi^{-1}(\widetilde{K})=K$. We fix $\widetilde{\Sigma}$ a smooth finite admissible rpcd (rational polyhedral cone decomposition) for $(GSp_6, \XX_6, \widetilde{K})$, and $\Sigma$ a smooth refinement of the finite admissible rpcd obtained by pullback $\varphi^*\widetilde{\Sigma}$. Then with respect to $\Sigma$, the integral model $\PPP_{K}$ admits a toroidal compactification $\PPP_{K}^{tor}$, which is a projective smooth scheme over $Spec(\OO_{E,(p)})$.  $\varphi$ induces an embedding of toroidal compactifications of integral models. The universal object $(\AAA,\lambda,i,\alpha)$ over $\PPP_{K}$ extends to $(\AAA^{ext},\lambda^{ext},i^{ext},\alpha^{ext})$ over $\PPP_{K}^{tor}$, where $\AAA^{ext}$ is a semi-abelian scheme over $\PPP_{K}^{tor}$, with certain addition structures.
	
\section{Automorphic bundles and dual BGG complexes}
	
	In this section, we give an explicit description of the dual BGG complex on the Picard modular surface. We mainly follow the results in \cite{BGG}. Fix a neat open compact subgroup $K\subset\GG(\hat{\ZZ})$ with hyperspecial $K_p$ from now on. For any $\OO_{E,(p)}$-algebra $R$, we write $\PPP_R:=\PPP_{K,R}$ and $\PPP_R^{tor}:=\PPP_{K,R}^{tor}$. Recall that our base ring is $R_1=\OO_{E,v}$, with a fractional field $E$ and a residue field $k_v=\FF_p$. We write $\PPP:=\PPP_{R_1}$, with generic fiber $P:=\PPP_E$ and special fiber $\overline{\PPP}:=\PPP_{k_v}$, and a similar notation for their toridal compactification. Let $(\AAA,\lambda,i,\alpha)$ be the universal object over $\PPP$. We have
	\begin{align}\label{hodge}
	0\longrightarrow Lie^{\vee}_{\AAA/\PPP}(1) \longrightarrow H^1_{dR}(\AAA/\PPP) \longrightarrow Lie_{\AAA^{\vee}/\PPP} \longrightarrow 0
	\end{align} 
	and
	\begin{align}\label{hodge_dual}
	0\longrightarrow Lie^{\vee}_{\AAA^{\vee}/\PPP}(1) \longrightarrow H_1^{dR}(\AAA/\PPP) \longrightarrow Lie_{\AAA/\PPP} \longrightarrow 0,
	\end{align}
	where $H_1^{dR}(\AAA/\PPP)$ is the dual vector bundle of $H^1_{dR}(\AAA/\PPP)$.
	
	There is an $\OO_E\otimes_\ZZ R_1$ action given by $i$ and a stucture map $R_1\to\OO_{\PPP}$, then again we have a type decomposition
	\begin{align}
	0\longrightarrow Lie^{\vee}_{\AAA^{\vee}/\PPP,\tau}(1) \longrightarrow H_1^{dR}(\AAA/\PPP)_{\tau} \longrightarrow Lie_{\AAA/\PPP,\tau} \longrightarrow 0 , \label{hodge_tau}\\
	0\longrightarrow Lie^{\vee}_{\AAA^{\vee}/\PPP,\bar{\tau}}(1) \longrightarrow H_1^{dR}(\AAA/\PPP)_{\bar{\tau}} \longrightarrow Lie_{\AAA/\PPP,\bar{\tau}} \longrightarrow 0 \label{hodge_bartau}
	\end{align}.
	Here $\tau$ is the natural embedding $\OO_E \to R_1$ and $\bar{\tau}$ is its complex conjugate. The element $a\otimes 1$ acts on (\ref{hodge_tau}) by $1\otimes \tau(a)$ and on (\ref{hodge_bartau}) by $1\otimes \bar{\tau}(a)$. The Lie algebra condition on the moduli problem implies that $Lie_{\AAA/\PPP,\tau}$ is of rank 2, and $Lie_{\AAA/\PPP,\bar{\tau}}$ is of rank 1.
	
	For any $R_1$-algebra $R$, let ${\rm Rep}_R(\GG_1)$ be the category of $R$-modules of finite presentations with algebra actions of $\GG_1$, where $\GG_1=\GG_{0,R_1}$. As defined in \cite[Def. 2.22]{BGG}, there is an exact functor
	\begin{align*}
	\E_{\GG_1,R}: {\rm Rep}_R(\GG_1) \rightarrow {\rm Coh}(\PPP_R).
	\end{align*}
	And $\E_{\PP_1,R},\E_{\MM_1,R}$ are defined similarly. If we view $W \in {\rm Rep}_R(\GG_1)$ as an object in ${\rm Rep}_R(\PP_1)$ by restriction, then $\E_{\GG_1,R}(W)=\E_{\PP_1,R}(W)$. If we view $W \in {\rm Rep}_R(\MM_1)$ as an object in ${\rm Rep}_R(\PP_1)$ by the quotient map $\PP_1\twoheadrightarrow \MM_1$, then $\E_{\MM_1,R}(W)=\E_{\PP_1,R}(W)$. For any $W \in {\rm Rep}_R(\GG_1)$, the coherent sheaf $\E_{\GG_1,R}(W)$ is called the \textit{automorphic sheaf} associated with $W$. It is called an \textit{automorphic bundle} if $W$ is locally free. 
	
	By definition, we have
	\begin{align*}
	&\E_{\GG_1,R}((L_0\oplus L_0^{\vee}(1))_R)=H_1^{dR}(\AAA/\PPP)_R  ,\\
	&\E_{\PP_1,R}(L_0^{\vee}(1)_R)=Lie^{\vee}_{\AAA^{\vee}/\PPP}(1)_R ,\qquad
	\E_{\PP_1,R}(L_{0,R})=Lie_{\AAA/\PPP,R},
	\end{align*}
	and the type decompositions agree.
	
	There is an extension version of (\ref{hodge}) on $\PPP^{tor}$:
	\begin{align}\label{hodge_ext}
	0\longrightarrow Lie^{\vee}_{\AAA^{ext}/\PPP^{tor}}(1) \longrightarrow H^1_{dR}(\AAA/\PPP)^{can} \longrightarrow Lie_{\AAA^{ext,\vee}/\PPP^{tor}} \longrightarrow 0,
	\end{align}
	where $H^1_{dR}(\AAA/\PPP)^{can}=\mathbb{H}^1(\Omega^{\bullet}_{\AAA^{ext}/\PPP^{tor}}(log\infty))$ is characterized and defined as \cite[Prop. 2.27]{BGG} and \cite[Prop 6.9]{toridal}. As defined in \cite[Def. 2.30]{BGG}, the exact sequence (\ref{hodge_ext}) induces an exact functor
	\begin{align*}
	\E_{\GG_1,R}^{can}: {\rm Rep}_R(\GG_1) \rightarrow {\rm Coh}(\PPP_R^{tor}).
	\end{align*}
	The sheaf $\E_{\GG_1,R}^{can}(W)$ is the \textit{canonical extension} of $\E_{\GG_1,R}(W)$ and $\E_{\GG_1,R}^{sub}(W):=\E_{\GG_1,R}^{can}(W)\otimes\mathcal{I}(D)$ is the \textit{subcanonical extension}, where $D$ is the complement of $\PPP$ in $\PPP^{tor}$ and $I(D)$ is its ideal sheaf. By definition, we have
	\begin{align*}
	&\E_{\GG_1,R}^{can}((L_0\oplus L_0^{\vee}(1))_R)=H_1^{dR}(\AAA/\PPP)^{can}_R  ,\\
	&\E_{\PP_1,R}^{can}(L_0^{\vee}(1)_R)=Lie^{\vee}_{\AAA^{ext,\vee}/\PPP^{tor}}(1)_R ,\qquad
	\E_{\PP_1,R}^{can}(L_{0,R})=Lie_{\AAA^{ext}/\PPP^{tor},R},
	\end{align*}
	and the type decompositions agree.
	
	Let $X:={\rm Hom}_{R_1}(\T_1, \GG_{\mathbf{m},R_1})$ be the character group of $\T_1$. Elements in $\T_1$ can be written as $(t_1,t_2,t_3,r)$, and elements in $X$ can be written as $(k_1,k_2,k_3,k_4)$. The element $(k_1,k_2,k_3,k_4)$ sends $(t_1,t_2,t_3,r)$ to $t_1^{k_1}t_2^{k_2}t_3^{k_3}r^{k_4}$. Let $\Phi_{\GG_1}\subset X$ (resp. $\Phi_{\MM_1}\subset \Phi_{\GG_1}$) be the subset of roots of $\GG_1$ (resp. $\MM_1$), and $\Phi_{\GG_1}^{+}$ (resp. $\Phi_{\MM_1}^{+}$) be the subset of positive roots (the choice is related to the choice of parabolic subgroup $\PP_1$), then 
	\begin{align*}
	&\alpha_1=(-1,1,0,0) ,\quad \alpha_2=(0,-1,1,0), \\
	&\Phi_{\GG_1}=\{\pm\alpha_1\,\pm\alpha_2,\pm(\alpha_1+\alpha_2)\} ,\\
	&\Phi_{\GG_1}^+=\{\alpha_1\,\alpha_2,\alpha_1+\alpha_2\} ,\\
	&\Phi_{\MM_1}=\{\pm\alpha_1\} ,\quad \Phi_{\MM_1}^+=\{\alpha_1\}.
	\end{align*}
	
	Let $X_{\GG_1}^+\subset X$(resp. $X_{\MM_1}^+$) be the subset of dominant weights of $\GG_1$(resp. $\MM_1$). The dominant condition of $\GG_1$ is $k_3 \ge k_2 \ge k_1$, and the dominant condition of $\MM_1$ is $k_2 \ge k_1$. The roots $\alpha_1,\alpha_2$ are simple positive roots and the corresponding reflections are $w_1(k_1,k_2,k_3,k_4)=(k_2,k_1,k_3,k_4)$, $w_2(k_1,k_2,k_3,k_4)=(k_1,k_3,k_2,k_4)$. Let $W_{\GG_1}$ be the Weyl group of $\GG_1$, and define
	\begin{align}
	W^{\MM_1}:=\left\lbrace w\in W_{\GG_1} \mid w(X_{\GG_1}^+)\subset X_{\MM_1}^+ \right\rbrace .
	\end{align}
	Then $W^{\MM_1}=\{Id, w_2, w_2w_1\}$. Let $\rho:=\frac{1}{2}\sum_{\alpha \in \Phi_{\GG_1}^+}\alpha=(-1,0,1,0)$. The \textit{dot action} of $W_{\GG_1}$ on $X$ is defined by $w\cdot\mu=w(\mu+\rho)-\rho$.
	
	For any $E$-algebra $R$ and dominant weight $\mu\in X_{\GG_1}^+$ (resp. $\mu\in X_{\MM_1}^+$), we write $V_{\mu,E}$ (resp. $W_{\mu,E}$) for the highest weight representation of $\GG_1$ (resp. $\MM_1$) over $E$ with the highest weight $\mu$, and $V_{\mu,R}=V_{\mu,E}\otimes_{E}R$ (resp. $W_{\mu,R}=W_{\mu,E}\otimes_{E}R$). It is not hard to compute the following.
	\begin{align*}
	L_{0,\tau,R} \cong W_{(0,1,0,1),R} ,\quad &L_{0,\bar{\tau},R} \cong W_{(0,0,-1,0),R} ,\\
	L_{0}^{\vee}(1)_{\tau,R} \cong W_{(0,0,1,1),R} ,\quad &L_{0}^{\vee}(1)_{\bar{\tau},R} \cong W_{(-1,0,0,0),R} ,\\
	L_{\tau,R} \cong V_{(0,0,1,1),R} ,\quad & L_{\bar{\tau},R} \cong V_{(-1,0,0,0),R}.
	\end{align*}
	
	\begin{prop}
		For any $E$-algebra $R$ and dominant weight $\mu=(k_1,k_2,k_3,k_4)\in X_{\MM_1}^+$, we have the description of the highest weight module
		\begin{align*}
		W_{\mu,R} \cong& {\rm Sym}^{k_2-k_1} (L_{0}^{\vee}(1)_{\bar{\tau},R}) \otimes {\rm det}(L_{0}^{\vee}(1)_{\bar{\tau},R})^{\otimes -k_2}   \notag  \otimes L_{0,\bar{\tau},R}^{\otimes (k_4-k_3)} \otimes L_{0}^{\vee}(1)_{\tau,R}^{\otimes k_4},
		\end{align*}
		and the associated vector bundle $\E_{\PP_1,R}(W_{\mu,R})$ is
		\begin{align}
		\E_{\PP_1,R}(W_{\mu,R})\cong& {\rm Sym}^{k_2-k_1} (Lie_{\AAA^{\vee}/\PPP}^{\vee}(1)_{\bar{\tau},R}) 
		\otimes {\rm det}(Lie_{\AAA^{\vee}/\PPP}^{\vee}(1)_{\bar{\tau},R})^{\otimes -k_2}   \notag  \\
		&\otimes Lie_{\AAA/\PPP,\bar{\tau},R}^{\otimes (k_4-k_3)} \otimes Lie^{\vee}(1)_{\AAA^{\vee}/\PPP,\tau,R}^{\otimes k_4}.
		\end{align}
		Its canonical extension in $\PPP_R^{tor}$ is
		\begin{align}
		\E^{can}_{\PP_1,R}(W_{\mu,R}) \cong& {\rm Sym}^{k_2-k_1} (Lie^{\vee}(1)_{\AAA^{ext,\vee}/\PPP^{tor},\bar{\tau},R}) \otimes {\rm det}(Lie^{\vee}(1)_{\AAA^{ext,\vee}/\PPP^{tor},\bar{\tau},R})^{\otimes -k_2}   \notag  \\
		&\otimes Lie_{\AAA^{ext}/\PPP^{tor},\bar{\tau},R}^{\otimes (k_4-k_3)} \otimes Lie^{\vee}(1)_{\AAA^{ext,\vee}/\PPP^{tor},\tau,R}^{\otimes k_4}
		\end{align}.
	\end{prop}
	
	\begin{rmk}
		By \cite[Prop. 1.17]{picard}, the structure of $\E_{\PP_1,R}(W_{\mu,R})$ as a vector bundle depends only on $k_1-k_2$ and $k_2-k_3$, which means
		\begin{align*}
		\E_{\PP_1,R}(W_{\mu,R})\cong& {\rm Sym}^{k_2-k_1} (Lie_{\AAA^{\vee}/\PPP}^{\vee}(1)_{\bar{\tau},R}) 
		\otimes {\rm det}(Lie_{\AAA^{\vee}/\PPP}^{\vee}(1)_{\bar{\tau},R})^{\otimes k_3-k_2}.
		\end{align*}
		The similar result also holds for $\E^{can}_{\PP_1,R}(W_{\mu,R})$. In a informal way, the choice of $k_4$ corresponds to normalizing the central character of modular forms. We sometimes choose $k_4=0$ for conveninence.
	\end{rmk}
	
	\begin{prop}
		For any $E$-algebra $R$ and dominant weight $\mu=(k_1,k_2,k_3,0)\in X_{\GG_1}^+$, we have the description of the highest weight module
		\begin{align*}
		V_{\mu,R} \cong {\rm Sym}^{k_2-k_1} (L_{\bar{\tau},R}) \otimes {\rm Sym}^{k_3-k_2} (\wedge^2(L_{\bar{\tau},R})) \otimes ({\rm det} L_{\bar{\tau},R})^{\otimes -k_3},
		\end{align*}
		and the associated vector bundle is
		\begin{align}
		\E_{\GG_1,R}(V_{\mu,R}) \cong {\rm Sym}^{k_2-k_1} (H^{dR}_{1,\bar{\tau},R}) \otimes {\rm Sym}^{k_3-k_2} (\wedge^2(H^{dR}_{1,\bar{\tau},R})) \otimes ({\rm det} H^{dR}_{1,\bar{\tau},R})^{\otimes -k_3}.
		\end{align}
		The similar result still holds for canonical extensions. 
	\end{prop}
	
	Now we apply Thm. 5.9 in \cite{BGG}, we have
	\begin{prop}
		Assume that $R$ is an $E$-algebra and $\mu\in X_{\GG_1}^+$. We have the dual BGG-complex on $\PPP_R^{tor}$
		\begin{align}
		BGG^{\bullet}(V_{\mu,R}^{\vee})^{can}=[\E_{\PP_1,R}^{can}(W_{\mu,R}^{\vee})\to \E^{can}_{\PP_1,R}(W_{w_2\cdot \mu,R}^{\vee})\to \E^{can}_{\PP_1,R}(W_{w_2w_1\cdot\mu,R}^{\vee})]
		\end{align}
		with a quasi-isomorphism 
		\begin{align}
		BGG^{\bullet}(V_{\mu,R}^{\vee})^{can}\hookrightarrow DR^{\bullet}(V_{\mu,R}^{\vee})^{can}:=\E_{\GG_1,R}^{can}(V_{\mu,R}^{\vee})\otimes\overline{\Omega}^{\bullet}_{\PPP_R^{tor}}.
		\end{align}
		The log-differential complex $\overline{\Omega}^{\bullet}_{\PPP_R^{tor}}$ is defined as $\wedge^{\bullet}\overline{\Omega}^1_{\PPP_R^{tor}}=\wedge^{\bullet}\Omega^{1}_{\PPP_R^{tor}}(log(D))$, and the differential operators of $ DR^{\bullet}(V_{\mu,R}^{\vee})^{can}$ are defined in \cite[Sect.2.4]{BGG} and \cite[Sect.4.2]{BGG}. The similar result holds if we replace canonical extensions with subcanonical extensions.
	\end{prop}
	
	We rewrite it in a more precise version. For $(\underline{k})\in X_{\GG_1}^+$, we say that it is \textit{cohomological} if $k_3\geq k_2+3$ and is \textit{good} if further $k_3\leq 2$ and $k_4=0$. For any cohomological weight we put
	\begin{align*}
	&\omega^{(\underline{k})}_R=\E_{\PP_1,R}(W_{(k_1,k_2,k_3,k_4),R}) ,\\
	&\mathcal{F}^{(\underline{k})}_R=\E_{\GG_1,R}(V_{(k_1+1,k_2+1,k_3-2,k_4),R}).
	\end{align*}
	
	As in $\overline{\Omega}^1_{\PPP_R^{tor}}=\E_{\GG_1,R}^{can}(\mathfrak{g_1}/\mathfrak{p_1})=\E_{\GG_1,R}^{can}(W_{(-1,0,1,0)})$, we rewrite the previous proposition as follows.
	
	\begin{prop}\label{BGG}
		Assume that $R$ is an $E$-algebra and a weight $(\underline{k})\in X_{\GG_1}^{+}$. We have 
		\begin{align}
		&BGG^{\bullet}(\mathcal{F}^{(\underline{k})}_R):=[\omega^{(k_2+1,k_3-2,k_1+1,k_4)}_R
		\to\omega^{(k_1,k_3-2,k_2+2,k_4)}_R \to \omega^{(k_1,k_2,k_3,k_4)}_R ]      \notag\\
		\hookrightarrow &DR^{\bullet}(\mathcal{F}^{(\underline{k})}_R):=[\mathcal{F}^{(\underline{k})}_R \to \mathcal{F}^{(\underline{k})}_R \otimes {\Omega}^{1}_{\PPP_R} \to \mathcal{F}^{(\underline{k})}_R \otimes {\Omega}^{2}_{\PPP_R}]
		\end{align} is a quasi-isomorphism, and a similar result holds for canonical and subcanonical extensions.
	\end{prop}
	
	We also define integral models of automorphic bundles.
	\begin{defn}
		For any $R_1$-algebra $R$ and dominant weight $(\underline{k})=(k_1,k_2,k_3,k_4)$, we put
		\begin{align}
		\omega_R^{(\underline{k})}=& {\rm Sym}^{k_2-k_1} (Lie_{\AAA^{\vee}/\PPP}^{\vee}(1)_{\bar{\tau},R}) 
		\otimes {\rm det}(Lie_{\AAA^{\vee}/\PPP}^{\vee}(1)_{\bar{\tau},R})^{\otimes -k_2}   \notag  \\
		&\otimes Lie_{\AAA/\PPP,\bar{\tau},R}^{\otimes (k_4-k_3)} \otimes Lie^{\vee}(1)_{\AAA^{\vee}/\PPP,\tau,R}^{\otimes k_4}.
		\end{align}
		If $k_4=0$ and $2\geq k_3 \geq k_2+3$, we put
		\begin{align}
		\mathcal{F}_R^{(\underline{k})}\cong {\rm Sym}^{k_2-k_1} (H^{dR}_{1,\bar{\tau},R}) \otimes {\rm Sym}^{k_3-k_2-3} (\wedge^2(H^{dR}_{1,\bar{\tau},R})) \otimes ({\rm det} H^{dR}_{1,\bar{\tau},R})^{\otimes -k_3+2}.
		\end{align}
		and similarly define the canonical and subcanonical extensions.
	\end{defn}
	
	\begin{defn}
		The space of the Picard modular forms of weight $(\underline{k})$ and level $K$ with coefficients in $R$ is defined as
		$$
		M_{(\underline{k})}(K,R):=H^0(\PPP_{K,R}^{tor}, \omega_R^{(\underline{k}),can}).
		$$
		The subspace of cusp forms is defined as
		$$
		S_{(\underline{k})}(K,R):=H^0(\PPP_{K,R}^{tor}, \omega_R^{(\underline{k}),sub}).
		$$
	\end{defn}

\section{Connectedness and monodromy groups of the non-basic Newton strata}
	
	In this section, we prove that for any non-basic Newton strata of $\overline{\PPP}_{K}$, it has exactly one connected component on each connected component of $\overline{\PPP}_{K}$, and compute the $l$-adic monodromy group associated with the $l$-adic Tate module of the universal abelian scheme over non-basic Newton strata. We follow the approach of C.-L. Chai and F. Oort in the Siegel case. Although Pol van Hoften \cite{van2024mod,van2024ordinary} has extended this approach to great generality, we include a proof in our more specific context to clarify the core ideas and ensure a self-contained account.
	
	We first recall the $l$-adic Hecke action. Let $k$ be an algebraic closure of $\FF_p$, and $\overline{\PPP}_{K}$ denote the special fiber of $\PPP_{K}$, where $K\subset \GG(\hat{\ZZ})$ is a neat open compact subgroup, and $K_p=\GG(\ZZ_p)$ hyperspecial. We fix $K_0$ such an open compact subgroup and let $\PPP_{K_0^l}:=(\PPP_{K_0^l N})_{N\subset K_{0,l}}$ be the $l$-adic tower of Shimura schemes over $\PPP_{K_0}$. And let $\overline{\PPP}_{K_0^l}:=(\overline{\PPP}_{K_0^l N})$ be the tower of special fibers. For any element $g\in \GG(\QQ_l)$ and $K=K_0^lN$, we have the following morphism:
	\begin{align*}
	\T_g: \overline{\PPP}_{K \cap g K_0 g^{-1}}(S) &\longrightarrow \overline{\PPP}_{g^{-1}Kg \cap K_0}(S)\\
	(A,\lambda,i,\alpha) &\mapsto (A,\lambda,i,\alpha g).
	\end{align*}    
    The morphism induces a $\GG(\QQ_l)$-action on the inverse limit $\PPP_{K_0^l}(S)$, and further on $\pi_0(\PPP_{K_0^l}):=\lim\limits_{\longleftarrow}\pi_0(\PPP_{K_0^l N})$. We fix a choice of the inverse system of connected components $\PPP^0_{K_0^l}=(\PPP^0_{K_0^l N})_{N\subset K_{0,l}}$, and their special fibers are geometrically connected, denoted by $\overline{\PPP}^0_{K_0^l}=(\overline{\PPP}^0_{K_0^l N})_{N\subset K_{0,l}}$. We have $\pi_0(\overline{\PPP}^0_{K_0^l})=\pi_0(\PPP^0_{K_0^l})=\pi_0(Sh^0_{K_0^l}(\CC))$, where 
	\begin{align*}
	Sh_{K_0^l}(\CC)=\lim_{\longleftarrow} \GG(\QQ)\backslash \XX \times \GG(\mathbb{A}_f)/K_0^lN=\GG(\QQ)\backslash \XX \times \GG(\mathbb{A}_f)/K_0^l.
	\end{align*}
	As $v:\pi_0(\GG(\QQ)\backslash \GG(\mathbb{A})/K_\infty K_0^lN)\cong\T(\QQ)\backslash \T(\mathbb{A})/v(K_\infty K_0^lN)$ for $N$ small enough, we have $Stab_{\GG(\QQ_l)}(\overline{\PPP}^0_{K_0^l})=Stab_{\GG(\QQ_l)}(Sh^0_{K_0^l}(\CC))=\ker v(\QQ_l)=\GG'(\QQ_l)$, here we view $\overline{\PPP}^0_{K_0^l}$ as an element of $\pi_0(\overline{\PPP}^0_{K_0^l})$ and similarly $Sh^0_{K_0^l}(\CC)$.
	
	We still use $\overline{\PPP}^0_{K}$ to denote its base change to $k$. The tower $\overline{\PPP}^0_{K_0^l}$ forms a pro-{\etale} covering over $\overline{\PPP}^0_{K_0}$, with the Galois group
	\begin{align*}
	\lim_{\longleftarrow} (K_{0,l} \cap \GG'(\QQ_l))/(N \cap \GG'(\QQ_l))=K_{0,l} \cap \GG'(\QQ_l)=:K'_{0,l}.
	\end{align*}
	
	For any geometric point $x\in \overline{\PPP}^0_{K_0}(k)$, we define $\mathcal{H}_l(x)$ as the projection of $\GG'(\QQ_l)\cdot\tilde{x}$, where $\tilde{x}\in \overline{\PPP}^0_{K_0^l}(k)$ is any lift of $x$. We say that a smooth locally closed subscheme $Z\subset \overline{\PPP}^0_{K_0}$ is $\GG'(\QQ_l)$-stable if $\mathcal{H}_l(Z(k))=Z(k)$. Moreover, we say $\GG'(\QQ_l)$ acts transitively on $\pi_0(Z)$ if for any connected component $Z^0$ of $Z$,  
	\begin{align*}
	\pi_0(\overline{\mathcal{H}_l(Z^0(k))} \cap Z) \to \pi_0(Z)
	\end{align*}
	is surjective. We write $\tilde{Z}:=Z\times_{\overline{\PPP}^0_{K_0}}\overline{\PPP}^0_{K_0^l}$ for the covering $(Z\times_{\overline{\PPP}^0_{K_0}}\overline{\PPP}^0_{K_0^lN})_{N\subset K_{0,l}}$, and $\pi_0(\tilde{Z}):=\lim\limits_{\longleftarrow}\pi_0(Z\times_{\overline{\PPP}^0_{K_0}}\overline{\PPP}^0_{K_0^lN}))$ is a profinite set. Now $Z$ is $\GG'(\QQ_l)$-stable if and only if $\tilde{Z}(k)$ is stable under the action of $\GG'(\QQ_l)$, and $\GG'(\QQ_l)$ acts transitively on $\pi_0(Z)$ if and only if it acts transitively on $\pi_0(\tilde{Z})$.

	Now we review some results on the Newton strata and Ekedahl-Oort strata. These two strata coincide in the Harris-Taylor case and in particular in our case $\overline{\PPP}_{K}$. There are several works on extending the EO-strata to toroidal and minimal compactifications, and we use the one on section 6 of Goldring and Koskivirta's paper \cite{Hasse}.
	\begin{defn}
		There exists a map
		\begin{align*}
		Nt:\overline{\PPP}_{K} \to Newt 
		\end{align*}
		such that for any $k$-point $x=(A,\lambda,i,\alpha)$ of $\overline{\PPP}_{K}$, $Nt(x)$ is the Newton polygon of $M(A[p^\infty])[1/p]$. Here $Newt$ is the set of all possible polygons, and $M(A[p^\infty])$ is the covariant Dieudonne module of $A[p^\infty]$. 
		The fibers of $Nt$ are locally closed subsets, and the Newton stratum attached to $\xi\in Newt$ is defined as $W_{\xi}=Nt^{-1}(\xi)$ with a reduced scheme structure. 
	\end{defn}
	
	There is a partial order on $Newt$: for any two Newton polygon $\xi_1,\xi_2 \in Newt$, $\xi_1<\xi_2$ if $\xi_1$ lies over $\xi_2$. The closure relation is $\overline{W}_{\xi}=\bigsqcup_{\zeta\leq\xi}W_{\zeta}$. The stratum corresponding to the maximal element is called the ordinary locus, and the one corresponding to the minimal one is called the basic locus.
	
	\begin{defn}
		There exists a smooth stack morphism given by the universal $G$-zip $\underline{I}$ 
		\begin{align*}
		\zeta_{K_0}: \overline{\PPP}_{K_0}\to G-Zip^{\mu}=[Z\backslash G].
		\end{align*}
		The Ekedahl-Oort stratum attached to $w\in \prescript{I}{}{W}$ (EO stratum for short) is defined as $\overline{\PPP}_{K_0,w}=(\zeta_{K_0})^{-1}([Z\backslash G_w])$, $\overline{\PPP}_{K_0,w}^*=(\zeta_{K_0})^{-1}([Z\backslash \overline{G}_w])$. $\prescript{I}{}{W}=W_I\backslash W$ 
	\end{defn}
	
	There is a partial order $"\prec"$ in $\prescript{I}{}{W}$ introduced by He in \cite{he2007}. And the closure relation in the EO strata is induced by $\overline{G}_{w}=\bigsqcup_{w'\preceq w} G_{w'}$.
	
	It is well known that the following proposition holds in the Harris-Taylor case and in particular our case, but does not hold in general.
	
	\begin{prop}\label{comparision}
		(1) There is an isomorphism of partially ordered sets between $(\prescript{I}{}{W}, \prec)$ and $(Newt,<)$, and both are totally ordered.\\
		(2) The Newton strata and the EO strata coincide in $\overline{\PPP}_K$. \\
		(3) The closure of each Newton stratum is smooth. \\
	\end{prop}
	
	  The result \cite[Thm 6.2.1]{Hasse} gives definitions of the extended EO-strata on toroidal and minimal compactifications.
	\begin{defn}
		The universal $G$-zip $\underline{I}$ of $\overline{\PPP}_{K_0}$ extends to a $\GG(\mathbb{A}_f^p)$-equivariant $G$-zip over $\overline{\PPP}_{K_0}^{tor}$. And the morphism $\zeta_{K_0}$ also extends to
		\begin{align*}
		\zeta_{K_0}^{tor}: \overline{\PPP}_{K_0}^{tor}\to G-Zip^{\mu}.
		\end{align*}
		Define $\overline{\PPP}_{K_0,w}^{tor}=(\zeta_{K_0}^{tor})^{-1}([Z\backslash G_w])$, $\overline{\PPP}_{K_0,w}^{tor,*}=(\zeta_{K_0}^{tor})^{-1}([Z\backslash \overline{G}_w])$ and $\overline{\PPP}_{K_0,w}^{min}=\pi(\overline{\PPP}_{K_0,w}^{tor})$, $\overline{\PPP}_{K_0,w}^{min,*}=\pi(\overline{\PPP}_{K_0,w}^{tor,*})$ for $w\in \prescript{I}{}{W}$. Here $\pi: \overline{\PPP}_{K_0}^{tor} \to \overline{\PPP}_{K_0}^{min}$ is the natural projection onto the minimal compactification.
	\end{defn}
	
	\begin{rmk}
		$\overline{\PPP}_{K_0,w}^{tor,*}$ might not be the Zariski closure of $\overline{\PPP}_{K_0,w}^{tor}$, and the analog of the closure relation is $\overline{\PPP}_{K_0,w}^{tor,*}=\bigsqcup_{w'\preceq w}\overline{\PPP}_{K_0,w'}^{tor}$.
	\end{rmk}
	
	\begin{prop}
		(1) The EO-strata $\overline{\PPP}_{K,w}$ (resp. $\overline{\PPP}_{K,w}^{tor}$) and $\overline{\PPP}_{K,w}^{*}$ (resp. $\overline{\PPP}_{K,w}^{tor,*}$) are of equal dimension of dimension $l(w)$\\
		(2) The stratum $\overline{\PPP}_{K_0,w}^{min}$ is an affine subvariety.\\
		(3) The stratum of dimension 0 $\overline{\PPP}_{K_0,e}^{tor}$ is equal to $\overline{\PPP}_{K_0,e}$.
	\end{prop}
	\begin{proof}
		The proofs are contained in \cite[Sect. 6]{Hasse}: (1) is an application of \cite[Cor. 6.4.4]{Hasse}, (2) is \cite[Prop. 6.3.1(c)]{Hasse}, and (3) is \cite[Lem. 6.4.1]{Hasse}.
	\end{proof}
	
	\begin{prop}\label{surjectivity}
		(1) $\overline{\PPP}_{K_0,w}$ is open and dense in $\overline{\PPP}_{K_0,w}^{min}$.\\
		(2) For $w\neq e$ and $Q\in\pi_0(\overline{\PPP}_{K_0,w})$, there exist $w'$ and $P\in\pi_0(\overline{\PPP}_{K_0,w'})$, such that $l(w')=l(w)-1$ and $P \subset \overline{Q}$, the Zariski closure of $Q$ in $\overline{\PPP}_{K_0}$.
	\end{prop}
	\begin{proof}
		For (1), it suffices to show $\dim \overline{\PPP}_{K_0,w}^{min} \cap D <\dim \overline{\PPP}_{K_0,w}=l(w)$, where $D=\overline{\PPP}_{K_0}^{min}-\overline{\PPP}_{K_0}$ is the boundary. If not, suppose that $w$ is the minimal length element that does not satisfy the condition. By (3) of the previous proposition, $w\neq e$ hence $l(w)>0$. As $\dim \overline{\PPP}_{K_0,w}^{min} \cap D =l(w)$, it contains an irreducible component $Y$ of dimension $l(w)$. Let $\overline{Y}$ denote the Zariski closure of $Y$ in $\overline{\PPP}_{K_0}^{min}$, then $\overline{Y}\subset \overline{\PPP}_{K_0,w}^{min,*}\cap D$. By (2) of the previous proposition, $Y$ is a positive dimensional affine subvariety in a projective variety, so $\partial Y=\overline{Y}\backslash Y$ is non-empty.
		
		As mentioned in \cite[Rmk. 6.3.2]{Hasse}, for PEL Shimura varieties of type A or C, the strata Hasse invariant in $\overline{\PPP}_{K_0,w}^{tor,*}$ descends to $\overline{\PPP}_{K_0,w}^{min,*}$. So, there exists
		$$
		h_w^{min}\in H^0(\overline{\PPP}_{K_0,w}^{min,*}, \omega^{N_w})
		$$ 
		whose nonvanishing locus is $\overline{\PPP}_{K_0,w}^{min}$, here $\omega$ is the hodge line bundle and $N_w$ a large integer. Pulling back to $\overline{Y}$, there exists $h_{\overline{Y}}\in H^0(\overline{Y},\omega|_{\overline{Y}}^{N_w})$, such that $\partial Y$ is the vanishing locus of it. Since $\partial Y$ is non-empty, we have $\dim \partial Y\geq l(w)-1$. The decomposition
		$$
		\partial Y=\bigsqcup\limits_{w'\prec w} (\partial Y\cap \overline{\PPP}_{K_0,w'}^{min})
		$$
		implies that there exists $w'\prec w$ such that 
		$$
		\dim D \cap\overline{\PPP}_{K_0,w'}^{min} \geq \dim \partial Y \cap \overline{\PPP}_{K_0,w'}^{min}\geq l(w)-1 \geq l(w').
		$$
		So $w'$ also does not satisfy the condition, and this contradicts the minimality of $l(w)$.
		
		For (2), we denote by $\overline{Q}^{min}$ the Zariski closure of $Q$ in $\overline{\PPP}_{K_0}^{min}$. Similarly, we have $\dim \overline{Q}^{min}\backslash Q=l(w)-1$, therefore, there exist a $w'$ and an irreducible component $P^{min}$ of $\overline{\PPP}_{K_0,w'}^{min}$, such that $\dim P^{min}=l(w')=l(w)-1$ and $ P^{min} \in \overline{Q}^{min}$. By (1) we know that $\dim P^{min}\cap D<\dim P^{min}$, so $P=P^{min}\cap \overline{\PPP}_{K_0}$ is what we want.
	\end{proof}

	Now we prove some results about connectedness and monodromy groups for some subvarieties, e.g. non-basic Newton stratum. We follow the idea of C.-L. Chai and F. Oort in \cite{chai2005monodromy} and \cite{2006Monodromy}.
	
	We fix an odd prime integer $l\neq p$ that splits in $E$, therefore $\GG_{\QQ_l}\cong (GL_3\times GL_1)_{\QQ_l}$ and $\GG'_{\QQ_l}\cong SL_{3,\QQ_l}$.
	\begin{lemma}\label{key lemma}
		Assume that $Z\subset\overline{\PPP}_{K_0}^0$ is a smooth locally closed subscheme. Choose $Z^0$ a connected component of $Z$ and $\bar{z}\in Z(k)$ a geometric point that lies above the genetric point of $Z^0$. Let $\rho_l:\pi_1^{\et}(Z^0,\bar{z}) \to \GG(\ZZ_l)={\rm Aut}(T_l\AAA_{\bar{z}},<,>)$ be the $l$-adic monodromy representation associated with $T_l(A_{Z^0})$, the Tate module of the universal abelian scheme over $Z^0$.
		If
		\begin{enumerate}
			\item $Z$ is $\GG'(\QQ_l)$-stable and $\GG'(\QQ_l)$ acts transitively on $\pi_0(Z)$ and
			\item ${\rm Im} \rho_l$ is not finite.
		\end{enumerate}
		then ${\rm Im}\rho_l=K^{'}_{0,l}$, and $Z$ is connected.
	\end{lemma}
	
	We apply this lemma to $Z_{\xi}$ and $\overline{Z_{\xi}}$ and prove the following theorem.
	
	\begin{thm}\label{connected}
		For any non-basic $\xi \in Newt$, we have $Z_{\xi}=W_{\xi} \cap \overline{\PPP}_{K_0}^{0}$ is connected, hence irreducible as it is smooth. And for $\rho_l: \pi_1^{\et}(Z_{\xi},\bar{z}) \to \pi_1^{\et}(\overline{Z_{\xi}},\bar{z}) \to \GG(\ZZ_l)$, the images of both fundamental groups in $\GG(\ZZ_l)$ are $K'_{0,l}$.
	\end{thm}
	
	To prove Lemma \ref{key lemma}, we need several results on algebraic groups.
	
	\begin{lemma} \label{alggp}
		(1) Let $M={\rm Im}\rho_l\subset\GG(\QQ_l)$ and $H:=\overline{M}^0\subset\GG_{\QQ_l}$. Then $H$ is semisimple, hence $H\subset\GG'_{\QQ_l}$. \\
		(2) The identity component of the normalizer $N_{\GG'_{\QQ_l}}(H)$ is a reductive algebraic group. \\
		(3) Let $\mathfrak{m}={\rm Lie}M,\mathfrak{h}={\rm Lie}H$, then $[\mathfrak{h},\mathfrak{h}]\subset \mathfrak{m}$.\\
		(4) If $Q$ is a finite index subgroup in $\GG(\QQ_l)$, then $Q=\GG(\QQ_l)$.
	\end{lemma}
	\begin{proof}
		By \cite[Cor. 1.3.9]{weil2} and \cite[Thm. 3.4.1]{weil2}, $\overline{M}$ is an extension of a finite group and a semisimple, hence $H$ is semisimple. 
		
		For (2), we consider the $H$-action on $\GG'_{\QQ_l}$ through conjugation and use $\cite[Thm. A.8.12]{conrad2015pseudo}$ to conclude that the identity component of $(\GG'_{\QQ_l})^{H}=Z_{\GG'_{\QQ_l}}(H)$ is connected reductive. The $N_{\GG'_{\QQ_l}}(H)$-action on $H$ through conjugation gives a $\QQ_l$-homomorphism
		$$
		f: N_{\GG'_{\QQ_l}}(H) \to {\rm Aut}_{H/\QQ_l}.
		$$
        For the identity component, we have $f^0:N_{\GG'_{\QQ_l}}(H)^0 \to {\rm Aut}_{H/\QQ_l}^0=H^{ad}$ and $f^0|_H$ is the quotient map. As $({\rm ker}f^0)^0=Z_{\GG'_{\QQ_l}}(H)^0$ is reductive, we conclude that $N_{\GG'_{\QQ_l}}(H)^0=H\cdot Z_{\GG'_{\QQ_l}}(H)^0$ is reductive. 
		
		(3) follows directly from \cite[Cor. 7.9]{borel1991linear}, as $\mathfrak{h}$ is the smallest algebraic Lie subalgebra containing $\mathfrak{m}$ in ${\rm Lie}(\GG_{\QQ_l})$.
		
		For (4), we consider the root subgroups. Fix a maximal split torus $\mathbb{T}$ of $\GG_{\QQ_l}$. For any root $\alpha$, let $U_{\alpha}$ be the root subgroup attached to $\alpha$. By assumption $Q\cap U_{\alpha}(\QQ_l)$ is a finite index subgroup in $U_{\alpha}(\QQ_l)$. As $U_{\alpha}$ is isomorphic to $\mathbb{G}_a$, the group $U_{\alpha}(\QQ_l)$ is divisible, hence $Q\cap U_{\alpha}(\QQ_l)=U_{\alpha}(\QQ_l)$. As $\GG'_{\QQ_l}$ is a split semisimple group, it is generated by its root subgroups, then $Q=\GG(\QQ_l)$. 
	\end{proof}
	
	Now we can give the proof of the Lemma \ref{key lemma}.
	
	\begin{proof}[Proof of Lemma \ref{key lemma}]
		Let $Z^0$ be a connected component of $Z$ and $\widetilde{Z^0}$ the fiber product $Z^0\times_{\overline{\PPP}^0_{K_0}} {\overline{\PPP}^0_{K_0^l}}$. Let $\widetilde{Z^0}^0$ be an inverse system of connected components of $Z^0\times_{\overline{\PPP}^0_{K_0}}{\overline{\PPP}^0_{K_0^lN}}$, it forms a pro-{\etale} covering over $Z^0$ and gives an element of $\pi_0(\widetilde{Z^0})$. 
		$$
		\begin{tikzcd}
		\widetilde{Z^0}^0 \arrow[hook,r] \arrow[rd] & \widetilde{Z^0} \arrow[hook,r] \arrow[d]& \widetilde{Z} \arrow[hook,r] \arrow[d]     & \overline{\PPP}_{K^l_0}^0 \arrow[d] \\
		& Z^0	\arrow[hook,r]	& Z	\arrow[hook,r]	& \overline{\PPP}_{K_0}^0	.		
		\end{tikzcd}
		$$
		As the right two squares are cartesian, we have
		$$
		{\rm Aut}_{Z^0}(\widetilde{Z^0})={\rm Aut}_{Z}(\widetilde{Z})={\rm Aut}_{\overline{\PPP}_{K_0}^0}(\overline{\PPP}_{K^l_0}^0)=K'_{0,l}.
		$$
		By definition we also have ${\rm Aut}_{Z^0}(\widetilde{Z^0}^0)={\rm Im}\rho_l=M$. Let $Q$ denote the stabilizer subgroup of $\widetilde{Z^0}^0$ in $\GG'(\QQ_l)$. Since the $K'_{0,l}$-action factor through the $\GG'(\QQ_l)$-action given by $l$-adic Hecke correspondence, $Q\cap K'_{0,l}=M$.
		
		By assumption $\GG(\QQ_l)$ acts transitively on $\pi_0(\widetilde{Z})$, so the base point $\widetilde{Z^0}^0$ gives bijections
		$$
		K'_{0,l}/M \xrightarrow{\sim} \pi_0(\widetilde{Z^0}) ,\qquad \GG'(\QQ_l)/Q \xrightarrow{\sim} \pi_0(\widetilde{Z}) .
		$$
		One can use the same argument as \cite[Lemma 2.8]{chai2005monodromy} to prove that the above two maps are homeomorphisms. Again by transitivity the profinite set $\pi_0(\widetilde{Z})$ is non-canonically homeomorphic to a disjoint union of finite copies of $\pi_0(\widetilde{Z^0})$.
		
		We claim that $N_{\GG'_{\QQ_l}}(H)^0$ is a parabolic subgroup in $\GG'_{\QQ_l}$. For any $\gamma\in Q$, there exists a small open subgroup $U\subset K'_{0,l}$ such that ${\rm Ad}_{\gamma}U\subset K'_{0,l}$, and then ${\rm Ad}_{\gamma}(U\cap M)\subset K'_{0,l}\cap Q=M$. Taking Lie algebras we get ${\rm Ad}_{\gamma}(\mathfrak{m})\subset\mathfrak{m}$, thus ${\rm Ad}_{\gamma}(H)\subset H$. Therefore $Q\subset N_{\GG'_{\QQ_l}}(H)(\QQ_l)$. Since $\GG'(\QQ_l)/Q \cong \pi_0(\widetilde{Z})$ is profinite hence compact, we conclude that the coset space $\GG'(\QQ_l)/N_{\GG'_{\QQ_l}}(H)^0(\QQ_l)$ is compact. By \cite[Prop. 9.3]{borel1965}, $N_{\GG'_{\QQ_l}}(H)^0$ is a parabolic subgroup. 
		
		In the other hand, we have proved $N_{\GG'_{\QQ_l}}(H)^0$ is reductive in Lemma \ref{alggp}(2), so $N_{\GG'_{\QQ_l}}(H)^0=\GG'_{\QQ_l}$. It implies that $H$ is a normal subgroup of the simple group $\GG'_{\QQ_l}$. By assumption $M$ is not finite thus $H$ is nontrivial, so we have $H=\GG'_{\QQ_l}$. Now Lemma \ref{alggp}(3) implies that $\mathfrak{m}=\mathfrak{h}={\rm Lie}\GG_{\QQ_l}$, hence $M$ contains an open subgroup of $\GG'(\QQ_l)$. 
		
		As $K'_{0,l}$ is open compact, $\pi_0(\widetilde{Z^0})\cong K'_{0,l}/M$ is finite and so as $\pi_0(\widetilde{Z})$. Moreover $\pi_0(\widetilde{Z^0})=\pi_0(\widetilde{Z})=1$ by Lemma \ref{alggp}(4). Hence ${\rm Im}\rho_l=M=K'_{0,l}$, and $Z$ is connected. 
	\end{proof}
	
	We also need the following result of F. Oort.
	\begin{lemma} \label{oort}
		Suppose that $k$ is a perfect field, $K$ is a function field of dimension 1 over $k$, and $X$ is an abelian variety over $K$. Let $l\neq char(k)$ be a prime number, and suppose that 
		$$
		\rho_l: Gal(K^s/K)\to {\rm Aut}(T_l(X))
		$$ 
		has the property that $\rho_l(Gal(K^s/Kk^s))$ is commutative. Then there exists a separable extension $L \supset K$, an abelian variety $Y$ over the algebraic closure of $k$ in $L$, and a purely inseparable isogeny
		$$
		t: Y\otimes L \to X\otimes_K L.
		$$
	\end{lemma}
	\begin{proof}
		\cite[Thm. 2.1]{Oort1974}.
	\end{proof}
	
	With these preparision, we can give the proof of Thm. \ref{connected}.
	
	\begin{proof}[Proof of Theorem \ref{connected}]
		It suffices to check that the two conditions in Lemma \ref{key lemma} hold for $Z_{\xi}:=W_{\xi} \cap \overline{\PPP}_{K_0}^{0}$, and then hold for $\overline{Z_{\xi}}$ automatically.
		
		For $\xi,\zeta\in Newt$, $\zeta<\xi$, we consider the map $i_{\xi}^{\zeta}:\pi_0(Z_{\zeta})\to \pi_0(Z_{\xi})$: $i_{\xi}^{\zeta}(P)=Q$ if and only if $P\subset \overline{Q}$. We have
		\begin{itemize}
			\item $i_{\xi}^{\zeta}$ is well-defined: For any $P\in \pi_0(Z_{\zeta})$, we have $P \subset Z_{\zeta}\subset \partial Z_{\xi}$. Then there exists $Q\in \pi_0(Z_{\xi})$ such that $\dim \partial Q\cap P=\dim P$, hence $P\subset \partial Q$ since $P$ is irreducible. The uniqueness of such $Q$ follows from the smoothness of $\overline{Z_{\xi}}$.
			\item $i_{\xi}^{\zeta}$ is Hecke-equivalent: By definition.
			\item $i_{\xi}^{\zeta}$ is surjective: By Prop. \ref{surjectivity} and Prop. \ref{comparision}.
		\end{itemize}
		We claim that $\GG'(\QQ_l)$ acts transitively on the basic locus $\pi_0(Z_{\sigma})$, hence it acts transitively on any $\pi_0(Z_{\xi})$. We have the following description of $\overline{\PPP}_{K,e}$
		$$
		\Theta_{K}: \overline{\PPP}_{K,e}(k) \cong I(\QQ) \backslash I(\mathbb{A}_f) / I_{p}^{e}K^{p}.
		$$
		Passing to inverse limit, we have
		$$
		\Theta_{K_0^l}: \overline{\PPP}_{K_0^l,e}(k) \cong I(\QQ) \backslash I(\mathbb{A}_f) / I_{p}^{e}K_0^{p,l}.
		$$
		The group $I$ is an inner form of $\GG_{\QQ}$ and $I_{p}^e$ is a maximal paraholic subgroup of $I(\QQ_p)=J_{b}(\QQ_p)$, we refer to \cite{terakado2022mass} for details. In our case
		$$
		I(\QQ_{v})=
		\left\{ \begin{array}{ll}
		GU(3,0) & v=\infty \\
		B^{\times}\times \QQ_p^{\times}   & v=p \\
		\GG(\QQ_v)   & v\neq p,\infty
		\end{array} \right.
		$$
		where $B$ is a division algebra over $\QQ_p$ of dimension 9 with ${\rm inv}B=1/3$, and $I_{p}^e=\OO_{B}^{\times}\times \ZZ_p^{\times}$. We denote by $I'$ the derived group of $I$ and recall the facts that $I'$ is simply connected and simple, $I'(\QQ_l)=\GG'(\QQ_l)$ is noncompact. Then by strong approximation theorem (cf. \cite[Thm. 7.12]{pr1994}), $I'(\QQ)$ is dense in $I'(\mathbb{A}^l))$, hence $I(\QQ)$ is dense in $I(\mathbb{A}_f^l)$. As $I_{p}^{e}K_0^{p,l}$ is open, $I(\QQ) \backslash I(\mathbb{A}_f^l) / I_{p}^{e}K_0^{p,l}$ is a singleton. Thus $I(\QQ) \backslash (I(\mathbb{A}_f^l)/I_{p}^{e}K^{p,l}\times I(\QQ_l))/ I(\QQ_l)$ is a singleton. Therefore $I(\QQ_l)=\GG(\QQ_l)$ acts transitively on $\overline{\PPP}_{K_0^l,e}(k)$, hence $\GG'(\QQ_l)$ acts transitively on $\overline{\PPP}_{K_0^l,e}\cap \overline{\PPP}_{K_0^l}^{0}=\widetilde{Z_{\sigma}}$.
		
		It remains to check the second condition in Lemma \ref{key lemma}. If ${\rm Im}\rho_l$ is finite, then there exists a finite {\etale} covering $p: \overline{C} \to \overline{Z_{\xi}^0}$, such that ${\rm Im}(\pi_1^{\et}(C,\overline{\eta}) \to \GG(\ZZ_l)$ is trivial, where $C=p^{-1}(Z_{\xi}^0)$, $\overline{\eta}$ is a geometric point with $p(\overline{\eta})=\overline{z})$, and the image is in the sense of the following diagram.
		
		$$
		\begin{tikzcd}
		\pi_1^{\et}(C,\overline{\eta}) \arrow[r] \arrow[d] & \pi_1^{\et}(\overline{C},\overline{\eta}) \arrow[d] & \\
		\pi_1^{\et}(Z_{\xi}^0, \overline{z}) \arrow[r] & \pi_1^{\et}(\overline{Z_{\xi}^0},\overline{z}) \arrow[r] & \GG(\ZZ_l).
		\end{tikzcd}
		$$
		
		Let $K=k(C)$ be the function field of $C$, and $\eta={\rm Spec}K\to C$ be the generic point of $C$. The image of
		$$
		Gal(K^s/K)=\pi_1^{\et}(\eta,\overline{\eta}) \to \pi_1^{\et}(C,\overline{\eta}) \to \GG(\ZZ_l)\to{\rm Aut}(T_l \AAA_{\eta})
		$$
		is trivial, then by Lemma \ref{oort} there exist a separable extension $L\supset K$, an abelian variety $Y$ over $k$, and a purely inseparable isogeny $Y_L\to \AAA_{\eta}\otimes_{K}L$. Let $\overline{C}_1$ be a finite {\etale} covering of $\overline{C}$ with function field $k(\overline{C}_1)=L$, and $\eta_1={\rm Spec}L\to \overline{C}_1$ be the generic fiber. As $Y_ {\overline{C}_1}$ and $\AAA_{\overline{C}_1}$ are Neron models of $Y_L$ and $\AAA_{\eta_1}$ respectively, the purely inseparable isogeny $Y_L\to \AAA_{\eta_1}$ extends to a purely inseparable isogeny $Y_{\overline{C}_1} \to \AAA_{\overline{C}_1}$. Thus ${\rm Im}(\overline{C}_1\to \overline{\PPP}_{K_0})=\overline{Z_{\xi}^0}$ lies in a single Newton stratum, which is impossible because of the surjectivity of $i_{\xi}^{\sigma}$.
	\end{proof}

\section{Rigid cohomology}
	In this section, we briefly recall the notion and results we need on the theory of rigid cohomology and refer to \cite{kiran2002} for details. Assume that $L$ is a finite extension of $\QQ_p$, $\OO_L$ the ring of integers, and $k_0$ the residue field. Formally, associated with a $k_0$-variety $X$ and an overconvergent $F$-isocrystal $\E$ (with respect to $L$), there exist rigid cohomology space $H^i_{rig}(X/L,\E)$ and rigid cohomology space with compact supports $H^i_{c,rig}(X/L,\E)$. These are finite dimensional $L$-vector spaces, which vanish for $i<0$ and $i>2\dim X$. If $X$ is smooth of pure dimension $n$, there is a canonical perfect pairing
	$$
	H^i_{rig}(X/L,\E)\times H^{2n-i}_{c,rig}(X/L,\E^{\vee})\to L(-n).
	$$
	For any closed subscheme $Z\subset X$, the rigid cohomology space with support in $Z$ is denoted by $H^i_{Z,rig}(X/L,\E)$. There is also a canonical perfect pairing
	$$
	H^i_{Z,rig}(X/L,\E)\times H^{2n-i}_{c,rig}(Z/L,\E^{\vee})\to L(-n).
	$$
	If further $Z$ is smooth of codimension $r$, we have a canonical Gysin isomorphism
	$$
	Gysin_{Z}: H^*_{rig}(Z/L,\E|_Z)\cong H^{*+2r}_{Z,rig}(X/L,\E)(r)
	$$
	where $(r)$ is the $r$-th Tate twist.
	
	Take $L=E_v=\QQ_p$ in our case. Let $\PPP_{K,\OO_{E_v}}^{tor}$ be the base change of $\PPP_{K}^{tor}$ via $R_1\hookrightarrow \OO_{E_v}$, $P_{K,E_v}^{tor}$ be its generic fiber, and $\overline{\PPP}_{K,\FF_p}^{tor}$ be the special fiber. We omit the subscript $\OO_{E_v}$, $E_v$, and $\FF_p$ when it is clear in this section. Let $X_{rig}$ (resp. $X_{an}$) denote the rigid space (resp. analytic space) associate with $X$. We have a natural specialization map $sp: P_{K,rig}^{tor} \to \overline{\PPP}_{K}^{tor}$.  For any locally closed subscheme $X\subset \overline{\PPP}_{K}^{tor}$, we write $]X[=sp^{-1}(X)$ for its tube in $P_{K,rig}^{tor}$. For any sheaf of abelian groups $\mathcal{F}$ defined on some strict neighborhood of $]X[$ in $]\overline{X}[$, and $j_{X}:X\to \overline{X}$ the open immersion, we put
	$$
	j_X^{\dagger}\mathcal{F}:=\lim\limits_{\longrightarrow} j_{V*}j_{V}^{*} \mathcal{F}
	$$
	where $V$ runs through a fundamental system of strict neighborhoods of $]X[$ in $]\overline{X}[$ on which $\mathcal{F}$ is defined, and $j_V:V\to ]\overline{X}[$ is the natural inclusion. Then For any overconvergent $F$-isocrystal $\E$ (with respect to $E_v$), the rigid cohomology is computed by
	$$
	R\Gamma_{rig}(X/E_v,\E):=R\Gamma(]\overline{X}[ ,j^{\dagger}_{X}DR^{\bullet}(\E)).
	$$
	To define rigid cohomology with supports, we define the functors
	\begin{align*}
	\underline{\Gamma}_{]X[}{\mathcal{F}}&:=\ker(j_X^{\dagger}\mathcal{F} \to i_*i^*\mathcal{F}) \\
	\underline{\Gamma}_{]Z[}^{\dagger}{\mathcal{F}}&:=\ker(j_X^{\dagger}\mathcal{F} \to j_{X-Z}^{\dagger}\mathcal{F})
	\end{align*}
	where $i:]\overline{X}-X[\to ]\overline{X[}$ is the canonical immersion, and $Z\subset X$ is a closed subscheme. Then the rigid cohomology with supports is computed by
	\begin{align*}
	R\Gamma_{c,rig}(X/E_v,\E)&:=R\Gamma(]\overline{X}[ , \underline{\Gamma}_{]X[}(DR^{\bullet}(\E))),\\
	R\Gamma_{Z,rig}(X/E_v,\E)&:=R\Gamma(]\overline{X}[ , \underline{\Gamma}_{]Z[}^{\dagger}(DR^{\bullet}(\E))).
	\end{align*}

	Let $\AAA$ denote the universal abelian variety over $\overline{P}_K$, and $H_{1}^{cris}(\AAA/\overline{\PPP}_K)$ denote the dual of relative crystalline cohomology. $H_{1}^{cris}(\AAA/\overline{\PPP}_K)$ is a $F$-crystal on the crystalline site $(\overline{\PPP}_K/\OO_{E_v})_{cris}$. As $P_{K,an}$ is a strict neighborhood of $]\overline{\PPP}_K[$, $H_{1}^{cris}(\AAA/\overline{\PPP}_K)$ gives an overconvergent $F$-isocrystal $\mathcal{D}$ with respect to $E_v$ on $\overline{\PPP}_K$. The $\OO_E$-action on $\AAA$ induces a decomposition $\mathcal{D}=\mathcal{D}_{\tau}\oplus\mathcal{D}_{\bar{\tau}}$. For a good weight $(\underline{k})=(k_1,k_2,k_3,0)$, $2\geq k_3 \geq k_2+3$, we put
	\begin{align}
	\mathcal{D}^{(\underline{k})}:= {\rm Sym}^{k_2-k_1} (\mathcal{D}_{\bar{\tau}}) \otimes {\rm Sym}^{k_3-k_2-3} (\wedge^2(\mathcal{D}_{\bar{\tau}})) \otimes ({\rm det} \mathcal{D}_{\bar{\tau}})^{\otimes -k_3+2}
	\end{align}
	By definition, $\mathcal{D}^{(\underline{k})}$ is exactly the rigid analytification of $\mathcal{F}^{(\underline{k})}_{E_v}$ as a sheaf on $P_{K,an}$. We still use $\mathcal{F}^{(\underline{k})}_{E_v}$ to denote its analytification. Then for any subscheme $X\subset \overline{\PPP}_K$, the rigid cohomology of $X$ with the coefficients in $\mathcal{D}^{(\underline{k})}$ is 
	$$
	R\Gamma_{rig}(X/E_v,\mathcal{D}^{(\underline{k})})=R\Gamma(]\overline{X}[,j^{\dagger}_{X}DR^{\bullet}(\mathcal{F}_{E_v}^{(\underline{k})}))
	$$
	
	Let $Y=\overline{\PPP}_K-\overline{\PPP}_K^{ord}$ denote the non-ordinary locus. Then there is a canonical distinguished triangle
	$$
	R\Gamma_{Y,rig}(\overline{\PPP}_K/E_v,\mathcal{D}^{(\underline{k})}) \to
	R\Gamma_{rig}(\overline{\PPP}_K/E_v,\mathcal{D}^{(\underline{k})}) \to
	R\Gamma_{rig}(\overline{\PPP}_K^{ord}/E_v,\mathcal{D}^{(\underline{k})}|_{\overline{\PPP}_K^{ord}}) \stackrel{+1}{\longrightarrow}
	$$
	and a canonical quasi-isomorphism called Gysin isomorphism
	$$
	Gysin_{Y}:R\Gamma_{rig}(Y/E_v,\mathcal{D}^{(\underline{k})}|_{Y}) \cong
	R\Gamma_{Y,rig}(\overline{\PPP}_K/E_v,\mathcal{D}^{(\underline{k})})[2](1).
	$$
	\begin{prop}
		There are canonical isomorphisms
		\begin{align*}
		H^*_{rig}(\overline{\PPP}_K/E_v, \mathcal{D}^{(\underline{k})}) \cong 
		&H^*_{dR}(P_{K}, (\mathcal{F}_{E_v}^{(\underline{k})},\nabla)) \cong
		\mathbb{H}^*(P_{K}^{tor}, DR^{\bullet}(\mathcal{F}_{E_v}^{(\underline{k})})^{can}) ,\\
		H^*_{c,rig}(\overline{\PPP}_K/E_v, \mathcal{D}^{(\underline{k})}) \cong 
		&H^*_{c,dR}(P_{K}, (\mathcal{F}_{E_v}^{(\underline{k})},\nabla)) \cong
		\mathbb{H}^*(P_{K}^{tor}, DR^{\bullet}(\mathcal{F}_{E_v}^{(\underline{k})})^{sub}).
		\end{align*}
	\end{prop}	
	\begin{proof}
		This proposition is well known. By \cite[Cor. 2.5]{BALDASSARRI199411} and \cite[Cor. 2.6]{BALDASSARRI199411}, we have
		$$
		H^*_{rig}(\overline{\PPP}_K/E_v, \mathcal{D}^{(\underline{k})}) \cong 
		H^*_{dR}(P_{K}, (\mathcal{F}_{E_v}^{(\underline{k})},\nabla)) \cong
		\mathbb{H}^*(P_{K,rig}, DR^{\bullet}(\mathcal{F}_{E_v}^{(\underline{k})})).
		$$
		As the natural embedding $j: P_{K,an} \hookrightarrow P^{tor}_{K,rig}$ is quasi-Stein, we have $R^ij_*\mathcal{F}=0$ for any coherent sheaf $\mathcal{F}$ and $i>0$, hence
		$$
		\mathbb{H}^*(P_{K,an}, DR^{\bullet}(\mathcal{F}_{E_v}^{(\underline{k})})) \cong \mathbb{H}^*(P_{K,rig}^{tor}, j_*DR^{\bullet}(\mathcal{F}_{E_v}^{(\underline{k})})).
		$$
		As explained in the proof of \cite[Prop. 5.17]{BGG}, after analytification, the natural embedding  $DR^{\bullet}(\mathcal{F}_{E_v}^{(\underline{k})})^{can} \hookrightarrow j_*DR^{\bullet}(\mathcal{F}_{E_v}^{(\underline{k})})$ is a quasi-isomorphism, hence
		$$
		\mathbb{H}^*(P_{K,rig}^{tor}, j_*DR^{\bullet}(\mathcal{F}_{E_v}^{(\underline{k})})) \cong \mathbb{H}^*(P_{K,rig}^{tor}, DR^{\bullet}(\mathcal{F}_{E_v}^{(\underline{k})})^{can}).
		$$
		
		The second row follows from the first row and Poincare duality. The duality of $\mathbb{H}^*(P_{K,rig}^{tor}, DR^{\bullet}(\mathcal{F}_{E_v}^{(\underline{k})})^{can})$ and $\mathbb{H}^{4-*}(P_{K,rig}^{tor}, DR^{\bullet}(\mathcal{F}_{E_v}^{(\underline{k}),\vee})^{sub})$ is explained, for example, in the proof of \cite[Lem. 4.11]{tx13}.
	\end{proof}
	
	\begin{defn}
		For any weight $(\underline{k})\in \ZZ^4$ that $k_2\geq k_1$, the space of overconvergent modular forms with coefficients in $E_v$ of level $K$ and weight $(\underline{k})$ is defined by
		$$
		S^{\dagger}_{(\underline{k})}(K,E_v):=H^0(P_{K,rig}^{tor}, j^{\dagger}_{\overline{\PPP}_K^{tor,ord}}\omega_{E_v}^{(\underline{k}),sub}),
		$$
		and for any extension $L/E_v$
		$$
		S^{\dagger}_{(\underline{k})}(K,L):=S^{\dagger}_{(\underline{k})}(K,E_v)\otimes_{E_v}L.
		$$
	\end{defn}
	
\section{Hecke actions and partial Frobenius}
	In this section, we recall the definition of prime-to-$p$ Hecke actions  and partial Frobenius on $\mathbb{H}^*(P^{tor}_{K,E_v}, j^{\dagger}_{\triangle}DR^{\bullet}(\mathcal{F}_{E_v}^{(\underline{k})})^{\Box})$ and $H^*(P_{K,rig}^{tor}, j^{\dagger}_{\triangle}\omega_{E_v}^{(\underline{k}),\Box})$, where $\triangle=\overline{\PPP}_K^{ord}$, $\overline{\PPP}_K^{tor,ord}$ or $\overline{\PPP}_K^{tor}$ and $\Box=sub$ or $can$.
	
	Let $\mathscr{H}(K^p,E_v)=E_v[K^p\backslash \GG(\AA_{f}^p)/K^p]$ denote the Hecke algebra. For any double coset $[K^p g K^p]$, we consider the following Hecke correspondence (here we choose suitable rational polyhedral cone decomposition data, we refer to \cite[Sect. 2]{toridal} for details):
	$$
	\begin{tikzcd}
	& \PPP_{K'}^{tor} \arrow{dl}[swap]{\pi_1=[1]^{tor}} \arrow{dr}{\pi_2=[g]^{tor}}& \\
	\PPP_{K}^{tor} &  & \PPP_{K}^{tor}
	\end{tikzcd}
	$$
	where $K'=K'^pK_p$ such that $K'^p=gK^pg^{-1}$, and $[g]^{tor}$ extends the natrual morphism $[g]:\PPP_{K'}^{tor}\to \PPP_{K}$ sending $(A,\lambda,i,\alpha)$ to $(A,\lambda,i,\alpha g)$.
	
	According to \cite[Thm. 2.15]{toridal}, we have canonical isomorphisms compatible with Hodge filtrations:
	\begin{align*}
		[g]^{tor,*} \mathcal{F}_{E_v}^{(\underline{k}),can} &\cong \mathcal{F}_{E_v}^{(\underline{k}),can} ,\\
		[g]^{tor,*} \omega_{E_v}^{(\underline{k}),can} &\cong \omega_{E_v}^{(\underline{k}),can} ,\\
		[g]^{tor,*} \OO_{P_K^{tor}}(-D) &\cong \OO_{P_{K'}^{tor}}(-D_{K'}).
	\end{align*}
	Hence, there are morphisms of abelian sheaves
	\begin{align*}
		\pi_2^*: \pi_2^{-1}\omega_{E_v}^{(\underline{k}),\Box} \to \omega_{E_v}^{(\underline{k}),\Box},	
	\end{align*}
	and morphisms of complexes
	\begin{align*}
		\pi_2^*:\pi_2^{-1}DR^{\bullet}(\mathcal{F}_{E_v}^{(\underline{k})})^{\Box} &\to DR^{\bullet}(\mathcal{F}_{E_v}^{(\underline{k})})^{\Box} ,\\
		\pi_2^*:\pi_2^{-1}BGG^{\bullet}(\mathcal{F}_{E_v}^{(\underline{k})})^{\Box} &\to BGG^{\bullet}(\mathcal{F}_{E_v}^{(\underline{k})})^{\Box}.
	\end{align*}
	which are compatible with Hodge filtrations and the natural quasi-isomorphic embedding $BGG^{\bullet}(\mathcal{F}_{E_v}^{(\underline{k})})^{\Box} \to  DR^{\bullet}(\mathcal{F}_{E_v}^{(\underline{k})})^{\Box}$
	
	By \cite[Lem. 3.3.12]{boxer2015torsion}, we have a natural trace map
	\begin{align*}
		Tr_{\pi_1}: \pi_{1,*} \OO_{P_{K'}^{tor}} \to \OO_{P_{K}^{tor}},
	\end{align*}
	and by \cite[Cor. 3.3.5]{boxer2015torsion}:
	\begin{align*}
		R^i\pi_1 \OO_{P_{K'}^{tor}}=0
	\end{align*}
	for $i>0$. Thus, by projection formula, we have
	\begin{align*}
		Tr_{\pi_1}&: R\pi_{1,*}DR^{\bullet}(\mathcal{F}_{E_v}^{(\underline{k})})^{\Box} \to DR^{\bullet}(\mathcal{F}_{E_v}^{(\underline{k})})^{\Box} ,\\
		Tr_{\pi_1}&: R\pi_{1,*}BGG^{\bullet}(\mathcal{F}_{E_v}^{(\underline{k})})^{\Box} \to BGG^{\bullet}(\mathcal{F}_{E_v}^{(\underline{k})})^{\Box} ,\\
		Tr_{\pi_1}&: R\pi_{1,*}\omega_{E_v}^{(\underline{k}),\Box} \to \omega_{E_v}^{(\underline{k}),\Box}	.
	\end{align*}
	
	Now we can describe the action of $[K^pgK^p]$ as:
	$$
	\begin{tikzcd}
		\mathbb{H}^2(P^{tor}_{K,E_v}, j^{\dagger}_{\triangle}DR^{\bullet}(\mathcal{F}_{E_v}^{(\underline{k})})^{\Box})
		\arrow{r}{\pi_2^*} \arrow{rd}[swap]{[K^pgK^p]}		
		& \mathbb{H}^2(P^{tor}_{K',E_v}, j^{\dagger}_{\triangle}DR^{\bullet}(\mathcal{F}_{E_v}^{(\underline{k})})^{\Box}) 
		\arrow{d}{Tr_{\pi_1}}\\	
		& \mathbb{H}^2(P^{tor}_{K,E_v}, j^{\dagger}_{\triangle}DR^{\bullet}(\mathcal{F}_{E_v}^{(\underline{k})})^{\Box}),
	\end{tikzcd}
	$$
	and
	$$
	\begin{tikzcd}
		H^0(P_{K,rig}^{tor}, j^{\dagger}_{\triangle}\omega_{E_v}^{(\underline{k}),\Box})
		\arrow{r}{\pi_2^*} \arrow{rd}[swap]{[K^pgK^p]}
		& H^0(P_{K',rig}^{tor}, j^{\dagger}_{\triangle}\omega_{E_v}^{(\underline{k}),\Box})
		\arrow{d}{Tr_{\pi_1}}\\
		& H^0(P_{K,rig}^{tor}, j^{\dagger}_{\triangle}\omega_{E_v}^{(\underline{k}),\Box}).
	\end{tikzcd}
	$$
	
	Then we define the partial Frobenius operators $Fr_{v}$ and $Fr_{\bar{v}}$ acting on the rigid cohomology space $H^*_{rig}(Z/E_v, \mathcal{D}^{(\underline{k})}|_{Z})$ (and similarly on $H^*_{c,rig}(Z/E_v, \mathcal{D}^{(\underline{k})}|_{Z})$) for $Z= \overline{\PPP}_{K}$, $\overline{\PPP}_{K}^{ord}$ or the non-ordinary locus $Y$, and consider the canonical subgroup of $\AAA^{ext}[p]$ in a strict neighborhood of $]\overline{\PPP}_K^{tor,ord}[$ to define the partial Frobenius on $\mathbb{H}^*(P^{tor}_{K,E_v}, j^{\dagger}_{\overline{\PPP}_{K}^{tor,ord}}DR^{\bullet}(\mathcal{F}_{E_v}^{(\underline{k})})^{\Box})$ and $H^*(P_{K,rig}^{tor}, j^{\dagger}_{\overline{\PPP}_{K}^{tor,ord}}\omega_{E_v}^{(\underline{k}),\Box})$.
	
	For any locally closed noetherian scheme $S/\FF_p$ and $S$-point $x=(A,\lambda,i,\alpha)$, we define $\phi_v(x)=(A'=A/Ker_v,\lambda',i',\alpha')$ where
	\begin{itemize}
		\item $Ker_v$ is the $v$-component of the kernal of the relative Frobenius $Fr_A:A\to A^{(p)}$. Let $\pi_{v}:A \to A/Ker_v$ denote the canonical isogeny.
		\item $i'$ is the $\OO_E$-action induced by $i$.
		\item $\lambda'$ is the polarization determined by the commutative diagram
		$$
			\begin{tikzcd}
				A' \arrow{r}{\check{\pi}_{v}} \arrow{d}{\cong}[swap]{\lambda'} & A \arrow{d}{\cong}[swap]{\lambda}	\\
				A'^{\vee} \arrow{r}{\pi_{v}^{\vee}}  & A^{\vee}
			\end{tikzcd}
		$$
		where $\check{\pi}_{v}$ is the unique map such that the composition
		\begin{tikzcd}
			A\arrow{r}{\pi_{v}} & A' \arrow{r}{\check{\pi}_{v}} & A
		\end{tikzcd}
		is the canonical quotient given by $A[v]$.
		\item $\alpha'$ is the one induced by $\alpha$ and $\pi_{v,*}: T^{(p)}A \to T^{(p)}A'$ the isomorphism of the prime-to-$p$ Tate modules.
	\end{itemize}

	Let $\phi_v:\overline{\PPP}_K\to\overline{\PPP}_K$ denote the induced endomorphism, which is finite flat of degree $p$ and $\phi(\overline{\PPP}_K^{ord})\subset \overline{\PPP}_K^{ord}$ and $\phi(Y)\subset Y$. For the universal abelian scheme $\AAA$ over $\overline{\PPP}_K$, we have an isogeny $\pi_{v}^*:\AAA \to \phi_v^*\AAA$, which induces an isomorphism of $F$-isocrystals:
	$$
		\pi_{v}^*: \phi_v^*\mathcal{D}^{(\underline{k})}\cong \mathcal{D}^{\underline{k}}.
	$$
	This gives rise to an operator $Fr_v$ acting on $H_{rig}^*(Z/E_v,\mathcal{D}^{(\underline{k})})$:
	$$
		\begin{tikzcd}
			H_{rig}^*(Z/E_v,\mathcal{D}^{(\underline{k})}|_{Z}) \arrow{r}{\phi_v^*} \arrow{rd}[swap]{Fr_v} &H_{rig}^*(Z/E_v,\phi_v^*\mathcal{D}^{(\underline{k})}|_{Z}) \arrow{d}{\pi_{v}^*}\\
			&H_{rig}^*(Z/E_v,\mathcal{D}^{(\underline{k})}|_{Z})
		\end{tikzcd}
	$$
	for $Z= \overline{\PPP}_{K}$, $\overline{\PPP}_{K}^{ord}$ or the non-ordinary locus $Y$. We can define the operator $Fr_{\bar{v}}$ in a similar way, and $Fr_vFr_{\bar{v}}=Fr$ is induced by $\phi_v\phi_{\bar{v}}=Fr_{\overline{\PPP}_K/\FF_p}$, the Frobenius endomorphism of $\overline{\PPP}_K$.
	
	Now we define the canonical subgroup of $\AAA^{ext}$, which lifts the kernal of Frobenius over a strict neighborhood of $]\overline{\PPP}_K^{tor,ord}[$. We apply \cite[Thm. 4]{fargues10} and use the same argument in \cite[Sect 5.2]{hernandez_2019} and \cite{Stroh2010} to deal with the boundary.
	
	Let $ha_{\tau}\in H^0(\overline{\PPP}_{K}^{tor}, \omega_{\tau}^{can, \otimes p-1})$ be the partial hasse invariant, whose non-vanishing locus is $\overline{\PPP}_K^{tor,ord}$, where $\omega_{\tau}=\det\pi_*\Omega^1_{\AAA/\overline{\PPP}_K,\tau}$. We define $ha_{\bar{\tau}}$ similarly and let $ha=ha_{\tau}\cdot ha_{\bar{\tau}}$. Let $\widetilde{ha_{\tau}} \in H^0(\PPP_{K}^{tor}, \omega_{\tau}^{can, \otimes p-1})$ be a lift of $ha_{\tau}$, and put
	\begin{align*}
		&Ha=(Ha_{\tau},Ha_{\bar{\tau}}): P_{K,rig}^{tor} \to [0,1]\times [0,1], \\
		Ha&(x)=({\rm inf}\{v_p(\widetilde{ha_{\tau}}(x)),1\}, {\rm inf}\{v_p(\widetilde{ha_{\bar{\tau}}}(x)),1\}).
	\end{align*}

	We define $P^{tor}_K(\epsilon)=Ha^{-1}([0,\epsilon_1)\times [0,\epsilon_2))$ for $\epsilon=(\epsilon_1,\epsilon_2)\in (0,1]\times(0,1]$. As $\epsilon$ tends to 0, $P^{tor}_K(\epsilon)$ forms a fundamental system of strict neighborhoods of $]\overline{\PPP}_K^{tor,ord}[$. We also denote by $P_K(\epsilon)=P^{tor}_K(\epsilon)\cap P_{K,an}$, it is an open dense subdomain of $P^{tor}_K(\epsilon)$. For $\epsilon$ small enough, one can define the canonical finite flat subgroup scheme $H\subset\AAA[p]$ over $P_K(\epsilon)$ by \cite[Thm. 4]{fargues10}, as $\AAA[p^{\infty}]$ is a Barsotti-Tate group of height 6, dimension 3.
	
	The problem is that $\AAA^{ext}$ is not a semi-abelian scheme of constant toric rank, and hence $\AAA^{ext}[p^\infty]$ is not finite flat over $P^{tor}_K(\epsilon)$. Based on the results in \cite[Sect. 3.1]{Stroh2010}, there exists an {\etale} covering $\overline{U}\to P^{tor}_{K,rig}$ together with {\etale} maps $pr_1,pr_2:\overline{R}\rightrightarrows \overline{U}$ such that
	$$
		P_{K,rig}^{tor}\simeq [\overline{U}/\overline{R}]
	$$
	and a Mumford 1-motive $M=[L\to \widetilde{G}]$ over $\overline{U}$ such that $M[p^n]=\AAA^{ext}[p^n]$. Here, $\widetilde{G}$ is a canonical semi-abelian scheme with an $\OO_E$-action of constant toric rank, thus $\widetilde{G}[p^n]$ is finite flat. Then we apply \cite[Thm. 4]{fargues10} to $\widetilde{G}[p]$, there exists a canonical finite flat subgroup scheme $H\subset\widetilde{G}[p]$ on $\overline{U}(\epsilon)=\overline{U}\times_{P_{K,rig}^{tor}}P_K^{tor}(\epsilon)$. The existence of this canonical subgroup that lies over $P_K(\epsilon)$ implies
	$$
		pr_1^*H = pr_2^*H
	$$
	over $R(\epsilon)=R\times_{P_{K,rig}}P_K(\epsilon)$. Over $\overline{R}$ we have the isomorphism 
	$$
		pr_1^*\AAA^{ext}\cong pr_2^*\AAA^{ext}
	$$
	by \cite[Thm. 3.1.5]{Stroh2010}. Thus, we have
	$$
		pr_1^*\AAA^{ext}/H\cong pr_2^*\AAA^{ext}/H
	$$
	over $R(\epsilon)$. As $\overline{R}$ is smooth and $R(\epsilon)$ is open dense in $R(\epsilon)$, the previous isomphism of semi-abelian schemes extends to an isomorphism over $\overline{R}(\epsilon)$ by \cite[Thm. 1.1.2]{Stroh2010}. Therefore, we see that $H$ descends to $P^{tor}_{K,rig}(\epsilon)$ by faithfully flat descent.
	
	The $\OO_E$-action on $H$ induces a decomposition $H=H_{v}\oplus H_{\bar{v}}$, and then the quotient isogeny $\pi_{v}: \AAA^{ext}\to \AAA^{ext}/H_v$ over $\PPP^{tor}_{K}(\epsilon)$ induces a finite flat map
	\begin{align*}
		\phi_v: \PPP^{tor}_{K}(\epsilon)\to \PPP^{tor}_{K}(\epsilon')
	\end{align*}
	such that $\phi_v^*\AAA^{ext}\cong\AAA^{ext}/H_v$ together with all induced structures. We also see that $\phi_v$ and $\pi_{v}$ induce maps of vector bundles.
	\begin{align*}
		\pi_{v}^*:\phi_v^* \mathcal{F}_{E_v}^{(\underline{k}),\Box} &\to \mathcal{F}_{E_v}^{(\underline{k}),\Box} \\
		\pi_{v}^*:\phi_v^* \omega_{E_v}^{(\underline{k}),\Box} &\to \omega_{E_v}^{(\underline{k}),\Box}
	\end{align*}
	over $P_{K}^{tor}(\epsilon)$, where $\Box$=$can$ or $sub$. And there is also a morphism of complexes:
	\begin{align*}
		\pi_{v}^*:\phi_v^* DR^{\bullet}(\mathcal{F}_{E_v}^{(\underline{k})})^{\Box} &\to DR^{\bullet}(\mathcal{F}_{E_v}^{(\underline{k})})^{\Box}.
	\end{align*}
	
	Now we can define the partial Frobenius operator $Fr_v$ for $\triangle=\overline{\PPP}_K^{ord}$, $\overline{\PPP}_K^{tor,ord}$ or $\overline{\PPP}_K^{tor}$ and $\Box=sub$ or $can$ as
	$$
	\begin{tikzcd}
		\mathbb{H}^*(P^{tor}_{K,E_v}, j^{\dagger}_{\triangle}DR^{\bullet}(\mathcal{F}_{E_v}^{(\underline{k})})^{\Box}) \arrow{rd}[swap]{Fr_v} \arrow{r}{\phi_{v}^*} &\mathbb{H}^*(P^{tor}_{K,E_v}, j^{\dagger}_{\triangle}\phi_v^*DR^{\bullet}(\mathcal{F}_{E_v}^{(\underline{k})})^{\Box}) \arrow{d}{\pi_{v}^*}\\
		&\mathbb{H}^*(P^{tor}_{K,E_v}, j^{\dagger}_{\triangle}DR^{\bullet}(\mathcal{F}_{E_v}^{(\underline{k})})^{\Box}),
	\end{tikzcd}
	$$
	and
	$$
	\begin{tikzcd}
	H^*(P_{K,rig}^{tor}, j^{\dagger}_{\triangle}\omega_{E_v}^{(\underline{k}),\Box}) \arrow{rd}[swap]{Fr_v} \arrow{r}{\phi_{v}^*} 
	&H^*(P_{K,rig}^{tor}, j^{\dagger}_{\triangle} \phi_v^*\omega_{E_v}^{(\underline{k}),\Box}) \arrow{d}{\pi_{v}^*},\\
	&H^*(P_{K,rig}^{tor}, j^{\dagger}_{\triangle}\omega_{E_v}^{(\underline{k}),\Box}),
	\end{tikzcd}
	$$
	and similarly $Fr_{\bar{v}}$.

\section{Main results}
    Now we are in the place to compute cohomology groups and to get our main results. For any cohomological weight $(\underline{k})$, let $\Theta: \omega_{E_v}^{(k_1,k_3-2,k_2+2,k_4),sub}\to \omega_{E_v}^{(k_1,k_2,k_3,k_4),sub}$ denote the differential operator in the BGG complex $BGG^{\bullet}(\mathcal{F}_{E_v}^{(\underline{k}),sub})$.
	\begin{prop} For any cohomological weight $(\underline{k})$, we have the commutative diagram
    $$
		\begin{tikzcd}
		S_{(\underline{k})}(K,L)  \arrow[d,hook] \arrow[r]
	    &  \dfrac{S^{\dagger}_{(\underline{k})}(K,L)}{\Theta(S^{\dagger}_{(k_1,k_3-2,k_2+2,0)}(K,L))} \arrow[d]
	    \\
        \mathbb{H}^2(P^{tor}_{K,L}, BGG^{\bullet}(\mathcal{F}_L^{(\underline{k})})^{sub})  \arrow{r} \arrow[d,"\cong"]
	    & \mathbb{H}^2(P^{tor}_{K,L}, j^{\dagger}_{\overline{\PPP}_K^{tor,ord}}BGG^{\bullet}(\mathcal{F}_L^{(\underline{k})})^{sub}) \arrow[d]
        \\
	    \mathbb{H}^2(P^{tor}_{K,L}, DR^{\bullet}(\mathcal{F}_L^{(\underline{k})})^{sub})  \arrow{r}
	    & \mathbb{H}^2(P^{tor}_{K,L}, j^{\dagger}_{\overline{\PPP}_K^{tor,ord}}DR^{\bullet}(\mathcal{F}_L^{(\underline{k})})^{sub}),
		\end{tikzcd}
    $$ where the horizontal maps are given by sheaf morphisms $\mathcal{F} \to j^{\dagger}_{\overline{\PPP}_K^{tor,ord}}\mathcal{F}$, the vertical maps from the first row to the second come from the definition of the cohomology of the complexes, and the vertical maps from the second row to the third come from the quasi isomorphism $BGG^{\bullet}(\mathcal{F}_L^{(\underline{k})})^{sub}\to DR^{\bullet}(\mathcal{F}_L^{(\underline{k})})^{sub}$.
	\end{prop}	
    
	\begin{proof}
        It suffices to show that the square consisting of the first two rows is commutative. Let $ _IE^{p,q}$ and $_{\II}E^{p,q}$ denote the spectral sequences defined by the cohomology of the complexes in the second row. We write $H^{\bullet}({ _{I}E})$ for the corresponding cohomology group with a filtration $Fil^{\bullet}$ such that
        $$
            Fil^{p}H^{p+q}( _{I}E)/Fil^{p+1}H^{p+q}( _{I}E)= {_{I}E_{\infty}^{p+q}},
        $$ and similarly for $H^{\bullet}({ _{\II}E})$.
        
        By \cite[Rmk. 5.24]{BGG}, the spectral sequence $ _IE^{p,q}$ degenerates on its first page and we have
        $$
            Fil^{2}H^{2}({ _{I}E})={ _{\II}E^{2,0}_{\infty}}= S_{(\underline{k})}(K,L).
        $$  
        Since the complex consists of three nonzero terms, we see that $_{\II}E^{p,q}$ degenerates on its third page and
        $$
            Fil^{2}H^{2}({ _{\II}E})={ _{\II}E^{2,0}_{\infty}}= { _{\II}E^{2,0}_{3}}= \coker( _{\II}E_{2}^{0,1}\longrightarrow {_{\II}E^{2,0}_{2}}).
        $$
        Therefore, we have the commutative diagram
        $$
            \begin{tikzcd}
                {_{I}E^{2,0}_{2}} \arrow[r] \arrow[d,equal]
                &{_{\II}E^{2,0}_{2}} \arrow[d] \\
                Fil^{2}H^{2}({ _{I}E}) \arrow[r] \arrow[d, hook]
                & Fil^{2}H^{2}({ _{\II}E})\arrow[d] \\
                H^{2}({ _{I}E}) \arrow[r]
                & H^{2}({ _{\II}E})
            \end{tikzcd} 
        $$ so the square consisting of the first two rows of the diagram in the proposition is commutative.
	\end{proof}
	
	We fix the isomorphisms $\iota_p:\overline{\QQ}_p\cong \CC$ and $\iota_l:\overline{\QQ}_l\cong\CC$. Let $\mathfrak{l}$ be the place corresponding to $E\hookrightarrow \CC\cong \overline{\QQ}_l$ via $\iota_l$, and $E_{\mathfrak{l}}$ be the completion of $E$ at $\mathfrak{l}$. Consider the dual of the universal abelian scheme $a:\AAA^{\vee} \to \PPP_{K}$ with $i:\OO_E\to {\rm End}(\AAA^{\vee})$ induced by the $\OO_E$-action of $\AAA$. We define $\mathcal{L}:=R^1 a_* E_{\mathfrak{l}}$, it is a lisse sheaf and admits a decomposition $\mathcal{L}=\mathcal{L}_{\tau} \oplus\mathcal{L}_{\bar{\tau}}$, where $\mathcal{L}_{\tau}$ is the direct summand on which $\OO_E$ acts via $\iota_l \circ \tau$. For any good weight $({\underline{k}})=(k_1,k_2,k_3,0)$ that $2\geq k_3 \geq k_2+3$, we define 
	\begin{align*}
	\mathcal{L}^{(\underline{k})}:= {\rm Sym}^{k_2-k_1} (\mathcal{L}_{\bar{\tau}}) \otimes {\rm Sym}^{k_3-k_2-3} (\wedge^2(\mathcal{L}_{\bar{\tau}})) \otimes ({\rm det} \mathcal{L}_{\bar{\tau}})^{\otimes -k_3+2}
	\end{align*}
	We have natural actions of $Fr_{\tau}$, $Fr_{\bar{\tau}}$, and the prime-to-$p$ Hecke algebra $\mathscr{H}(K^p, E_{\mathfrak{l}})$ on $H^*_{\et}(Y_{\overline{\FF}_p}, \mathcal{L}^{(\underline{k})})$ and $H^*_{c,\et}(Y_{\overline{\FF}_p}, \mathcal{L}^{(\underline{k})})$. Then we have the following comparison proposition.
	
	\begin{prop}
		We identify $\overline{\QQ}_p$ and $\overline{\QQ}_l$ with $\CC$ via isomorphisms $\iota_p$ and $\iota_l$. Then we have an equality in the Grothendieck groups of finite-dimensional $\mathscr{H}(K^p,\CC)[Fr_{\tau}, Fr_{\bar{\tau}}]$-modules:
		\begin{align}
		\sum_{i=0}^{2}(-1)^{i}[H^i_{c,rig}(Y/E_v,\mathcal{D}^{(\underline{k})}) \otimes_{E_v} \overline{\QQ}_p]=\sum_{i=0}^{2}(-1)^{i}[H^i_{c,\et}(Y_{\overline{\FF}_p},\mathcal{L}^{(\underline{k})}) \otimes_{E_{\mathfrak{l}}} \overline{\QQ}_l].
		\end{align}
		Moreover, we have
		\begin{align}
		[H^2_{c,rig}(Y/E_v,\mathcal{D}^{(\underline{k})}) \otimes_{E_v} \overline{\QQ}_p]=[H^2_{c,\et}(Y_{\overline{\FF}_p},\mathcal{L}^{(\underline{k})}) \otimes_{E_{\mathfrak{l}}} \overline{\QQ}_l]
		\end{align}
		and by Poincare duality
		\begin{align}
		[H^0_{rig}(Y/E_v,\mathcal{D}^{(\underline{k})}) \otimes_{E_v} \overline{\QQ}_p]=[H^0_{\et}(Y_{\overline{\FF}_p},\mathcal{L}^{(\underline{k})}) \otimes_{E_{\mathfrak{l}}} \overline{\QQ}_l].
		\end{align}
	\end{prop}
	
	\begin{proof}
		Note that $\mathcal{D}^{(\underline{k})}$ and $\mathcal{L}^{(\underline{k})}$ are pure of weight $w=(k_2-k_1)+2(k_3-k_2-3)+3(-k_3+2)=-(k_1+k_2+k_3)$. Then by Deligne's Weil \uppercase\expandafter{\romannumeral2} and its rigid analog (we refer to \cite[Thm. 6.6.2]{kiran2002}), $H^i_{c,rig}(Y/E_v,\mathcal{D}^{(\underline{k})})$ and $H^i_{c,\et}(Y_{\overline{\FF}_p},\mathcal{L}^{(\underline{k})})$ are both mixed of weight $\leq w+i$. We also have $H^2_{c,rig}(Y/E_v,\mathcal{D}^{(\underline{k})})$ and $H^2_{c,\et}(Y_{\overline{\FF}_p},\mathcal{L}^{(\underline{k})})$ both pure of weight $w+2$, so the second part of the proposition follows immediately from the first part.
		
		To prove the first part, we follow the idea in \cite[Sect. 6]{saito_2009}. Consider the decomposition $E\otimes_{\QQ} E\cong E_{\tau}\oplus E_{\bar{\tau}}$, where $E_{\tau}$ is the copy of $E$ with embedding $\tau$ and similarly $E_{\bar{\tau}}$. As we have $\delta\in \OO_{E,(p)}^{\times}$, the projection $e_{\bar{\tau}}=1\otimes 1- \delta \otimes \delta^{-1}$ induces a linear combination of prime-to-$p$ endomorphisms $e_{\bar{\tau}}=i(1)\otimes 1- i(\delta) \otimes \delta^{-1}$ on the dual of the universal abelian scheme $\AAA^{\vee}/\PPP_{K,\OO_{E,(p)}}$, where $i:\OO_E \to {\rm End}(\AAA^{\vee})$ is induced by the $\OO_E$-action of $\AAA$. Then $e_{\bar{\tau}}$ acts on $\mathcal{L}=R^1a_* E_{\mathfrak{l}}$ as an idempotent that $e_{\bar{\tau}}\cdot \mathcal{L}=\mathcal{L}_{\bar{\tau}}$.
		
		Using the Vandermonde determinant, one can easily find a $\QQ$-linear combination $e^1$ of multiplications by prime-to-$p$ integers acting on $\AAA$, such that the induced action of $e^1$ on $R^1a_*\QQ_l$ is the identity and is equal to $0$ on $R^qa_*\QQ_l$ for $q\neq 1$. Consider the fiber product $a^{w}: \AAA^{\vee,w} \to \PPP$, then $e_{\bar{\tau}}^{\otimes w}\cdot (e^1)^{\otimes w}$ acts as an idempotent on $R^q a^{w}_* E_{\mathfrak{l}}$, and we get $\mathcal{L}_{\bar{\tau}}^{\otimes w}$ if $q=w$ and $0$ if $q\neq w$. Therefore, there exists an element $e_w\in \QQ[S_w]$ in the group algebra of the symmetric group $S_w$, such that $e_w\cdot \mathcal{L}_{\bar{\tau}}^{\otimes w}=\mathcal{L}^{(\underline{k})}$.
		
		Taking a product, we finally get a $L$-linear combination $e^{(\underline{k})}$ of algebraic correspondences on $\AAA^{\vee,w}$ such that:
		\begin{itemize}
			\item It is an $E$-linear combination of permutations in $S_w$ and prime-to-$p$ endomorphisms of $\AAA^{\vee,w}$
			\item The action of $e^{(\underline{k})}$ on $R^qa_*^w E_{\mathfrak{l}}$ is the projection onto the direct summand $\mathcal{L}^{(\underline{k})}$ if $q=w$ and is equal to 0 if $q\neq w$.
		\end{itemize}
		Then $e^{(\underline{k})}$ also acts on $H^q_{c,\et}(\AAA^{\vee}, E_{\mathfrak{l}})$. Considering the Leray spectral sequence for $a^w$,  we have
		$$
		e^{(\underline{k})}\cdot H^{i+w}_{c,\et}(\AAA^{\vee,w} \times_{\overline{\PPP}_K} Z_{\overline{\FF}_p}, E_{\mathfrak{l}})=H^{i}_{c,\et}(Z_{\overline{\FF}_p}, \mathcal{L}^{(\underline{k})})
		$$
		for any locally closed subscheme $Z\subset \overline{\PPP}_K$.
		
		Similarly, let $(\AAA^{\vee,w}/\OO_{E_v})_{cris}$ over $(\overline{\PPP}_K/\OO_{E_v})_{cris}$ denote the small crystalline site and $H^q_{cris}(\AAA^{\vee,w}/\overline{\PPP}_K)$ denote the relative crystalline cohomology. This is an $F$-crystal on $(\overline{\PPP}_K/\OO_{E_v})_{cris}$, and we denote by $H^q_{rig}(\AAA^{\vee,w}/\overline{\PPP}_K)$ the associated overconvergent $F$-isocrystal on $\overline{\PPP}_K$. Again we have $e^{(\underline{k})}$ acting on $H^q_{rig}(\AAA^{\vee,w}/\overline{\PPP}_K)$ as an idempotent, and
		\begin{align*}
		e^{(\underline{k})}\cdot H^q_{rig}(\AAA^{\vee,w}/\overline{\PPP}_K)=
		\left\lbrace \begin{array}{ll}
		\mathcal{D}^{(\underline{k})} & q=w\\
		0 & q\neq w
		\end{array}	 		
		\right..
		\end{align*}
		Then we have
		\begin{align*}
		e^{(\underline{k})}\cdot H^{i+w}_{c,rig}(\AAA^{\vee,w} \times_{\overline{\PPP}_K} Z / E_v)=H^{i}_{c,rig}(Z/ E_v, \mathcal{L}^{(\underline{k})}).
		\end{align*}
		
		By the same argument in the proof of \cite[Prop. 4.13]{tx13}, the actions of $Fr_p$, $S_p$ and the prime-to-$p$ Hecke operators $[K^pgK^p]$ acting on $H^{i}_{c,\et}(Z_{\overline{\FF}_p}, \mathcal{L}^{(\underline{k})})$ (resp. $H^{i}_{c,rig}(Z/ E_v, \mathcal{L}^{(\underline{k})})$) can be realized as compositions of $e^{(\underline{k})}$ and the actions on $H^{i+w}_{c,\et}(\AAA^{\vee,w} \times_{\overline{\PPP}_K} Z_{\overline{\FF}_p}, E_{\mathfrak{l}})$ (resp. $H^{i+w}_{c,rig}(\AAA^{\vee,w} \times_{\overline{\PPP}_K} Z / E_v)$) induced by some algebraic correspondences on $\AAA^{\vee,w}$. 
		
		Therefore, to prove the first part of the proposition, it suffices to show that, for any algebraic correspondence $\Gamma$ of $\AAA^{\vee,w}$, we have the equality:
		\begin{align*}
		&\sum_{i}(-1)^{i}{\rm Tr}(\Gamma^*, H^{i}_{c,\et}(\AAA^{\vee,w} \times_{\overline{\PPP}_K} Y_{\overline{\FF}_p}, E_{\mathfrak{l}}))\\
		=&\sum_{i}(-1)^{i}{\rm Tr}(\Gamma^*, H^{i}_{c,rig}(\AAA^{\vee,w} \times_{\overline{\PPP}_K} Y/E_v)).
		\end{align*}
		And this follows from \cite[Cor. 3.3]{mieda2009} 
	\end{proof}
	
	We consider the $p$-adic local system $\mathscr{L}=R^1 a_* \QQ_p$ on $P_{K,E}$, where $a:\AAA^{\vee}\to P_{K,E}$ is the structure map. Similarly, for a good weight $(\underline{k})=(k_1,k_2,k_3,0)$, $2\geq k_3 \geq k_2+3$, we define
	\begin{align*}
	\mathscr{L}^{(\underline{k})}:= {\rm Sym}^{k_2-k_1} (\mathscr{L}_{\bar{\tau}}) \otimes {\rm Sym}^{k_3-k_2-3} (\wedge^2(\mathscr{L}_{\bar{\tau}})) \otimes ({\rm det} \mathscr{L}_{\bar{\tau}})^{\otimes -k_3+2}.
	\end{align*}

	Assume that $L/\QQ_p$ is a finite extension and $f\in S_{(\underline{k})}(K, L)$ is a cuspidal eigen-newform (i.e. $f$ is an eigenform and $\dim (\pi^{\infty})^{K}=1$, where $\pi(f)=\pi^{\infty}\otimes\pi_{\infty}$ is the associated cuspidal automorphic representation). Let $\mathfrak{m}$ denote the maximal ideal $\mathscr{H}(K^p,L)$ associated with the Hecke eigensystem of $f$. Assume that the associated Galois representation $\sigma(f)=(H^2_{c,\et}(P_{K,\overline{E}}, \mathscr{L}^{(\underline{k})})\otimes_{\QQ_p}L)[\mathfrak{m}]$ is irreducible of dimension 3. We write $\sigma_p=\sigma(f)|_{{\rm Gal}(\overline{\QQ}_p/\QQ_p)}$ and $D_p=D_{cris}(\sigma_p)$.
	
	\begin{prop}
		We have the exact sequence 
		$$
		0=H^2_{Y, rig}(\overline{\PPP}_K/L,\mathcal{D}^{(\underline{k})}))[\mathfrak{m}]\to H^2_{rig}(\overline{\PPP}_K/L,\mathcal{D}^{(\underline{k})})[\mathfrak{m}]\to 
		H^2_{rig}(\overline{\PPP}^{ord}_K/L,\mathcal{D}^{(\underline{k})})[\mathfrak{m}].
		$$
	\end{prop}
	
	\begin{proof}
		By Thm. \ref{connected}, we have 
		$$
		H^0_{\et}(Y_{\overline{\FF}_p}, \mathcal{L})=(T^{K'_l})_{E_{\mathfrak{l}}}^{\oplus\#\pi_0(Y)}
		$$ 
		where $T\cong T_l(\AAA_x)$ is an algebraic $\GG_{\ZZ_l}$-representation and $x$ is a geometric point.
		Then we have
		\begin{align*}
		H^0_{\et}(Y_{\overline{\FF}_p},\mathcal{L}^{(\underline{k})})&= (T_{(k_1+1,k_2+1,k_3-2,0),\OO_{E_{\mathfrak{l}}}}^{K'_l})_{E_{\mathfrak{l}}}^{\oplus\#\pi_0(Y)}\\
		&
		=\left\lbrace \begin{array}{ll}
		0 & (\underline{k})\neq (k-1,k-1,k+2,0)\\
		(\det^{\otimes k})^{\oplus\#\pi_0(Y)} & (\underline{k})= (k-1,k-1,k+2,0)
		\end{array}
		\right.
		\end{align*}
		since $T_{(k_1+1,k_2+1,k_3-2,0),\OO_{E_{\mathfrak{l}}}}$ is a $\GG(\OO_{E_{\mathfrak{l}}})$-irreducible module and $K'_l$ is an open compact subgroup of $\GG'(\OO_{E_{\mathfrak{l}}})$. Here $\det$ comes from the character
		\begin{align*}
		\GG(E)\cong (GL_3\times GL_1)(E) \to& GL_1(E) \\
		(g,r)\longmapsto& \det g.
		\end{align*}
		
		Considering the Hecke eigenvectors associated with $\mathfrak{m}$, we have $H^0_{\et}(Y_{\overline{\FF}_p}, \mathcal{L})[\mathfrak{m}]=0$ (otherwise $\sigma(f)$ factor through a character). By the previous proposition and the Gysin isomorphism, we have
		\begin{align*}
		0=H^0_{rig}(Y/E_v, \mathcal{D}^{(\underline{k})})[\mathfrak{m}]\cong H^2_{Y,rig}(\overline{\PPP}_K/E_v, \mathcal{D}^{(\underline{k})})[\mathfrak{m}]
		\end{align*} 
	\end{proof}
	
	By \cite[Thm 1.1]{llz2019}, we have the commutative diagram
	$$
	\begin{tikzcd}
	H^2_{c,dR}(P_{K,L},(\mathcal{F}_L^{(\underline{k})},\nabla))\otimes_{L}B_{dR} \arrow[d] \arrow[r,"\sim"] &H^2_{c,\et}(P_{K,\overline{\QQ}_p},\mathscr{L}^{(\underline{k})})\otimes_{\QQ_p}B_{dR} \arrow[d] \\
	H^2_{dR}(P_{K,L},(\mathcal{F}_L^{(\underline{k})},\nabla))\otimes_{L}B_{dR} \arrow[r,"\sim"] &H^2_{\et}(P_{K,\overline{\QQ}_p},\mathscr{L}^{(\underline{k})})\otimes_{\QQ_p}B_{dR} \\
	\end{tikzcd}.
	$$
	Then by the irreducibility of $\sigma(f)$ we have
	$$
	\begin{tikzcd}
	H^2_{c,dR}(P_{K,L},(\mathcal{F}_L^{(\underline{k})},\nabla))[\m] \arrow[d,hook] \arrow[r,"\sim"] &\mathbb{H}^2(P^{tor}_{K,L}, DR^{\bullet}(\mathcal{F}_L^{(\underline{k})})^{sub})[\m] \arrow[d,hook] \\
	H^2_{dR}(P_{K,L},(\mathcal{F}_L^{(\underline{k})},\nabla))[\m] \arrow[r,"\sim"]  &\mathbb{H}^2(P^{tor}_{K,L}, DR^{\bullet}(\mathcal{F}_L^{(\underline{k})})^{can})[\m]
	\end{tikzcd}.
	$$

	In summary, for any good weight $(\underline{k})$ and finite extension $L$ over $E_v=\QQ_p$, we have the following commutative diagram. 
	$$
	\begin{tikzcd}	 	
	S_{(\underline{k})}(K,L)[\m]  \arrow[d,hook] \arrow[r]
	&  \dfrac{S^{\dagger}_{(\underline{k})}(K,L)[\m]}{\Theta(S^{\dagger}_{(k_1,k_3-2,k_2+2,0)}(K,L)[\m])} \arrow[d]
	\\
	\mathbb{H}^2(P^{tor}_{K,L}, DR^{\bullet}(\mathcal{F}_L^{(\underline{k})})^{sub})[\m] \arrow[d,hook] \arrow{r}
	& \mathbb{H}^2(P^{tor}_{K,L}, j^{\dagger}_{\overline{\PPP}_K^{tor,ord}}DR^{\bullet}(\mathcal{F}_L^{(\underline{k})})^{sub})[\m] \arrow{d}
	\\
	\mathbb{H}^2(P^{tor}_{K,L}, DR^{\bullet}(\mathcal{F}_L^{(\underline{k})})^{can})[\m] \arrow{r} \arrow{d}{\cong}
	& \mathbb{H}^2(P^{tor}_{K,L}, j^{\dagger}_{\overline{\PPP}_K^{tor,ord}}DR^{\bullet}(\mathcal{F}_L^{(\underline{k})})^{can})[\m] \arrow[d]
	\\
	H^2_{rig}(\overline{\PPP}_K/\QQ_p, \mathcal{D}^{(\underline{k})})[\m] \otimes_{\QQ_p} L \arrow[r, hook]
	&
	H^2_{rig}(\overline{\PPP}_K^{ord}/\QQ_p, \mathcal{D}^{(\underline{k})})[\m]\otimes_{\QQ_p} L
	\end{tikzcd}.
	$$
	We can conclude as follows.
	\begin{prop}\label{injectivity}
		For any good weight $(\underline{k})$ and finite extension $L$ over $\QQ_p$, we have the commutative diagram
		$$
		\begin{tikzcd}	 	
		S_{(\underline{k})}(K,L)[\m]  \arrow[d,hook] \arrow[r,hook]
		&  \dfrac{S^{\dagger}_{(\underline{k})}(K,L)[\m]}{\Theta(S^{\dagger}_{(k_1,k_3-2,k_2+2,0)}(K,L)[\m])} \arrow[d,hook]
		\\
		\mathbb{H}^2(P^{tor}_{K,L}, DR^{\bullet}(\mathcal{F}_L^{(\underline{k})})^{sub})[\m] \arrow[r,hook,"i"]
		& \mathbb{H}^2(P^{tor}_{K,L}, j^{\dagger}_{\overline{\PPP}_K^{tor,ord}}DR^{\bullet}(\mathcal{F}_L^{(\underline{k})})^{sub})[\m] 
		\end{tikzcd}
		$$
		and $i$ is $\varphi$-equivalent, where the $\varphi$-action on $\mathbb{H}^2(P^{tor}_{K,L}, DR^{\bullet}(\mathcal{F}_L^{(\underline{k})})^{sub})$ comes from the operator $Fr\otimes 1$ on $H^2_{c,rig}(\overline{\PPP}_K/\QQ_p,\mathcal{D}^{(\underline{k})})\otimes_{\QQ_p}L$, and is equal to $Fr\otimes 1$ on the right hand side.
	\end{prop}
	
	\begin{rmk}
		Twisting central characters, it is obvious that the previous proposition holds for arbitrary cohomological weight $(\underline{k})$.
	\end{rmk}

	We further assume that $\sigma_p$ is crystalline and that the characteristic polynomial of $\varphi$ on $D_p$ splits in $L$ with three different roots. Now we describe the filtered $\varphi$-module $D_p$. 
	\begin{align*}
	D_p&=(H^2_{c,\et}(P_{K,\overline{\QQ}_p}, \mathscr{L}^{(\underline{k})})[\m]\otimes_{\QQ_p}B_{dR})^{{\rm Gal}(\overline{\QQ}_p/\QQ_p)}\otimes_{\QQ_p}L\\ 
	&\cong H^2_{c,dR}(P_{K,\QQ_p},(\mathcal{F}_{\QQ_p}^{(\underline{k})}, \nabla))[\m]\otimes_{\QQ_p} L\\ 
	&\cong H^2_{c,rig}(\overline{\PPP}_K/\QQ_p,\mathcal{D}^{(\underline{k})})[\m]\otimes_{\QQ_p} L.\\
	\end{align*}
	The operator $\varphi$ comes from the operator $\varphi\otimes 1$ on $H^2_{c,rig}(\overline{\PPP}_K/\QQ_p,\mathcal{D}^{(\underline{k})})[\m]\otimes_{\QQ_p} L$. The Frobenius semilinear operator $\varphi$ is linear because the rigid cohomology is over $\QQ_p$. And the filtration of $D_p$ comes from the Hodge filtration on the de Rham cohomology (we refer to \cite[Sect. 3.3]{BGG} for the Hodge filtration of the de Rham complexes). For convenience, we write the BGG complex as 
	$$
	BGG^{\bullet}(\mathcal{F}_L^{(\underline{k})})^{sub}=[\omega_1^{sub}\to\omega_2^{sub}\to\omega_3^{sub}]
	$$
	then it is not hard to compute the Hodge filtration as:
	\begin{align*}
	Fil^a D_p
	&=\mathbb{H}^2(P^{tor}_{K,L}, Fil^a DR^{\bullet}(\mathcal{F}_L^{(\underline{k})})^{sub})[\m] \\
	&=\left\lbrace 
	\begin{array}{ll}
	\mathbb{H}^2(P^{tor}_{K,L}, BGG^{\bullet}(\mathcal{F}_L^{(\underline{k})})^{sub})[\m] 
	& a\leq k_1+1 \\
	\mathbb{H}^1(P^{tor}_{K,L}, [\omega_2^{sub}\to \omega_3^{sub}])[\m]
	& k_1+1 < a \leq k_2+2 \\
	H^0(P^{tor}_{K,L},\omega_3^{sub})[\m]
	& k_2+2 < a \leq k_3 \\
	0 & k_3<a
	\end{array}
	\right. ,\\
	Gr^aD_p
	&=\left\lbrace 
	\begin{array}{ll}
	H^2(P^{tor}_{K,L},\omega_1^{sub})[\m] & a=k_1+1 \\
	H^1(P^{tor}_{K,L},\omega_2^{sub})[\m] & a=k_2+2 \\
	H^0(P^{tor}_{K,L},\omega_3^{sub})[\m] & a=k_3 \\
	0 & otherwise
	\end{array}
	\right..
	\end{align*} 
	
	As $\varphi$ is $L$-linear on $D_p$, the characteristic polynomial $char(\varphi)$ makes sense. By assumption $char(\varphi)=(x-\alpha_1)(x-\alpha_2)(x-\alpha_3)$ for some $\alpha_i\in L$ and $v_p(\alpha_1) \leq v_p(\alpha_2) \leq  v_p(\alpha_3)$. Then there exists a basis $e_1,e_2,e_3$ that satisfies
	$$
	D_p=Le_1 \oplus Le_2 \oplus Le_3 ,\quad \varphi(e_i)=\alpha_i e_i.
	$$
	Let $D_i=Le_i$ and $D_{ij}=Le_i\oplus Le_j$ denote the $\varphi$-stable subspaces, equipped with induced filtrations.
	
	\begin{thm}\label{phi-stable}
		$Fil^{k_3}D_p$ is $\varphi$-stable if and only if $\varphi f-\alpha_3 f= \Theta(g)$ for some cuspidal overconvergent form $g\in S^{\dagger}_{(k_1,k_3-2,k_2+2,0)}(K,L)[\m]$, moreover,
		$$
		k_1+1 \leq v_p(\alpha_1) \leq v_p(\alpha_2)\leq k_2+2<v_p(\alpha_3)=k_3
		$$
		hence $D_3$ and $D_{12}$ are admissible. Here we view $f$ as an overconvergent form, and thus $\varphi f$ makes sense.
	\end{thm}
	
	\begin{proof}
		We claim that if $Fil^{k_3}D_p$ is $\varphi$-stable, then $\varphi$=$\alpha_3\cdot id$ on it, and 
		$$
		k_1+1 \leq v_p(\alpha_1) \leq v_p(\alpha_2)\leq k_2+2<v_p(\alpha_3)=k_3
		$$
		hence $D_{12}$ and $D_3$ are admissible. By the admissibility of $D_p$, we have (cf. \cite[Sect. 2]{ghate2009})
		\begin{align*}
		t_H(D_p)&=t_N(D_p) \\
		k_1+1\leq t_H(D_1)&\leq t_N(D_1) = v_p(\alpha_1) \\
		(k_1+1)+(k_2+2)\leq t_H(D_{12})&\leq t_N(D_{12}) =v_p(\alpha_1)+v_p(\alpha_2) \\
		k_3=t_H(D')&\leq t_N(D')\leq v_p(\alpha_3)
		\end{align*}
		where $D'=Fil^{k_3}D_p$ equipped with the induced filtration. Then
		$$
		k_1+1 \leq v_p(\alpha_1) \leq v_p(\alpha_2)\leq k_2+2<v_p(\alpha_3)=k_3
		$$
		and $D'=D_3$, hence $D_3$ and $D_{12}$ are admissible.
		
		By Prop. \ref{injectivity}, we have
		$$
		\begin{tikzcd}	 	
		Fil^{k_3}D_p=L\cdot f  \arrow[d,hook] \arrow[r,hook]
		&  \dfrac{S^{\dagger}_{(\underline{k})}(K,L)[\m]}{\Theta(S^{\dagger}_{(k_1,k_3-2,k_2+2,0)}(K,L)[\m])} \arrow[equal,d]
		\\
		D_p	\arrow[r,hook,"i"]
		& \mathbb{H}^2(P^{tor}_{K,L}, j^{\dagger}_{\overline{\PPP}_K^{tor,ord}}DR^{\bullet}(\mathcal{F}_L^{(\underline{k})})^{sub})[\m] 
		\end{tikzcd}.
		$$
		Then $Fil^{k_3}D_p$ is $\varphi$-stable if and only if $\varphi f-\alpha_3 f=0$ in $\dfrac{S^{\dagger}_{(\underline{k})}(K,L)[\m]}{\Theta(S^{\dagger}_{(k_1,k_3-2,k_2+2,0)}(K,L)[\m])}$, and hence if and only if $\varphi f-\alpha_3 f= \Theta(g)$ in $S^{\dagger}_{(\underline{k})}(K,L)[\m]$ for some $g$.
	\end{proof}
	
	Then, twisting by a central character, we immediately have the following corollary.
	\begin{cor} \label{split}
		Assume that $f\in S_{(\underline{k})}(K,L)$ is of a cohomological weight $(\underline{k})$. Then the crystalline Galois representation $\sigma_p$ whose Hodge-Tate weight is $(k_1-k_4+1,k_2-k_4+2,k_3-k_4 )$ splits as $\sigma_p=\sigma_{p,12}\oplus\sigma_{p,3}$, whose Hodge-Tate weights are $(k_1-k_4+1,k_2-k_4+2)$ and $(k_3-k_4)$ respectively, if and only if $(\varphi-\alpha_3)f=\Theta(g)$ for some $g\in S^{\dagger}_{(k_1,k_3-2,k_2+2,k_4)}(K,L)[\m]$. In particular, $g$ is a companion form of $f$.
	\end{cor}


	\bibliographystyle{alpha}

\end{document}